\documentclass[11pt,letterpaper]{article}
\usepackage[utf8]{inputenc}
\pdfoutput=1

\usepackage{fullpage}
\usepackage{amsmath}
\usepackage{palatino}
\usepackage{macros}
\usepackage{latexsym} 
\usepackage{mathtools}
\usepackage{bbm}
\usepackage{authblk}
\usepackage[export]{adjustbox}

\usepackage{tikz}
\usetikzlibrary{arrows,positioning, shapes.geometric}

\usepackage{enumitem}

\usepackage{thmtools} 
\usepackage{thm-restate}

\def\eps{\epsilon}

\def\cW{{\cal W}}
\def\cC{{\cal C}}

\def\prt{S}
\def\b1{{\bf 1}}
\def\W{\mathbb{W}}\def\cI{{\cal I}}

\def\dimension{d}
\allowdisplaybreaks

\title{Trickle-down Theorems via C-Lorentzian Polynomials II: Pairwise Spectral Influence and Improved Dobrushin's Condition}

 \author[1]{Jonathan Leake}
 \affil{\small University of Waterloo, \textsf{jonathan.leake@uwaterloo.ca}}
 \author[2]{Shayan Oveis Gharan}
  \affil{\small University of Washington,  \textsf{shayan@cs.washington.edu}}

\begin{document}
\maketitle
\thispagestyle{empty}
\begin{abstract}
Let $\mu$ be a probability distribution on a multi-state spin system on a set $V$ of sites. Equivalently, we can think of this as a $d$-partite simplicial complex with distribution $\mu$ on maximal faces. 
    For any pair of vertices $u,v\in V$, define the pairwise spectral influence $\cI_{u,v}$ as follows. Let $\sigma$ be a choice of spins $s_w\in S_w$ for every $w\in V \setminus \{u,v\}$, and construct a matrix in $\R^{(S_u\cup S_v)\times (S_u\cup S_v)}$ where for any $s_u\in S_u, s_v\in S_v$, the $(us_u,vs_v)$-entry is the probability that $s_v$ is the spin of $v$ conditioned on $s_u$ being the spin of $u$ and on $\sigma$. Then $\cI_{u,v}$ is the maximal second eigenvalue of this matrix, over all choices of spins for all $w \in V \setminus \{u,v\}$. Equivalently, $\cI_{u,v}$ is the maximum local spectral expansion of links of codimension $2$ that include a spin for every $w \in V \setminus \{u,v\}$.

    We show that if the largest eigenvalue of the pairwise spectral influence matrix  with entries $\cI_{u,v}$ is bounded away from 1, i.e. $\lambda_{\max}({\cI})\leq 1-\eps$ (and $X$ is connected), then the Glauber dynamics mixes rapidly and generate samples from $\mu$.
    This improves/generalizes the classical Dobrushin's influence matrix as the $\cI_{u,v}$ lower-bounds the classical influence of $u\to v$. As an application, we prove that the Glauber dynamics mixes rapidly up to (approximately) the phase transition for the multi-state hard core model -- a widely studied model in telecommunication networks and statistical physics (generalizing the hard core model) introduced by Mazel and Suhov. As a by-product of our results, we also prove improved/almost optimal trickle-down theorems for partite simplicial complexes.
    % Our proof builds on the recent trickle-down theorems via $\cC$-Lorentzian polynomials machinery.
    Our proof builds on the trickle-down theorems via $\cC$-Lorentzian polynomials machinery recently developed by the authors and Lindberg.
\end{abstract}

\newpage
\pagestyle{empty}
\tableofcontents
\newpage
\pagestyle{plain}\setcounter{page}{1}
\section{Introduction}
A simplicial complex $X$ on a finite ground set $[n]=\{1,\dots,n\}$
is a downwards closed collection of subsets of $[n]$, i.e., if $\tau\in X$ and $\sigma \subset \tau$, then $\sigma \in X$. 
The elements of $X$ are called {\bf faces}, and the maximal faces are called {\bf facets}. 
We say that a face $\tau$ is of dimension $k$ if $|\tau| = k$ and write $\dimfn(\tau) = k$\footnote{Note that this differs from the typical topological definition of dimension for faces of a simplicial complex.}. %\textcolor{blue}{(JL: typically the dimension is annoyingly 1 less than the size.)\textcolor{red}{Sh: That is ok I think we can shift by 1 to make our life easier}}
A simplicial complex $X$ is a {\bf pure} $\dimension$-dimensional complex if every facet  has dimension $\dimension$. 
%In this paper, all  simplicial complexes are assumed to be pure. 

Given a  $\dimension$-dimensional complex $X$,
for any   $0 \leq i \leq \dimension$,    define  $X(i) = \{\tau \in X : \dimfn(\tau) = i\}$.  Moreover, the codimension of a face $\tau \in X$ is defined as $\codim (\tau)  = \dimension - \dimfn (\tau)$. For a face $\tau \in X$, define the  {\bf link} of $\tau$ as the simplicial complex $X_{\tau} = \{\sigma \setminus \tau : \sigma \in X, \sigma \supset \tau\}$.
%Note that $X_\tau$ is a $(\codim(\tau) -1)$- dimensional complex.

A  $\dimension$-partite  complex is a $\dimension$-dimensional complex such that $X(1)$ can be (uniquely) partitioned into sets $\prt_1\cup \dots \cup \prt_{\dimension}$ such that for  every facet  $\tau \in X (\dimension)$,  we have $|\tau \cap \prt_u|=1$ for all $u \in V$. Let $\mu$ be a probability distribution on facets of $X$. We denote this by $(X,\mu)$. 

\paragraph{Multi-state Spin Systems.} Equivalently, a $d$-partite simplicial complex can be seen as a multi-state spin system. In a general multi-state spin system, we have
\begin{itemize}
    \item a set of $d$ sites which we denote by $V$,
    \item a site $v$ is associated with a set of spins $S_v$, so that $X(1)=\cup_{v \in V} S_v$, and
    \item each possible configuration of the system is an element of the set $S_1\times \dots \times S_d$ and $\mu$ is a probability distribution over all configurations.
\end{itemize}
We emphasize that in general the sets $\{S_v\}_{v\in V}$ may even be exponentially large in $V$ and they do not need to be the same.

Many of the classical models in statistical physics such as the Ising and the Potts model can be seen as special cases of such a system. Thus, in the rest of this manuscript we will use these two terminologies interchangeably.

\subsection{Dobrushin's Condition}

Given a $d$-partite complex $(X,\mu)$, a site $u\in V$, and a choice of spin $s_v \in S_v$ for all $v\neq u$, a {\bf marginal oracle} samples the spin of $u$ from $\mu$ conditioned on the spin of all the other sites. Given a such an oracle, one can run a Markov chain called the Glauber dynamics (or the down-up walk) to generate random samples from $\mu$: Each time choose $u \in V$ uniformly at random, and re-sample the spin of $u$ given all the other spins using the marginal oracle. Under certain connectivity assumptions, this Markov chain converges to $\mu$. So, a natural question is to find sufficient conditions for the Markov chain to mix rapidly. Dobrushin's condition is then a simple sufficient condition for rapid mixing.

To state Dobrushin's condition, we need to develop a little more notation. For a face $\tau\in X(i)$, let $\mu_\tau$ denote the conditional distribution of $\mu$ on  facets of the link $X_\tau$, i.e, 
$$\mu_\tau(\omega) = \P_{\sigma\sim\mu}[\sigma\setminus\tau=\omega \mid \tau\subset\sigma] \quad \text{for all} \quad \omega\in X_\tau(d-i).$$
Given two probability distributions $\mu$ and $\nu$, we write $\norm{\mu-\nu}_{TV}$ to denote the total variation distance between $\mu$ and $\nu$.

\paragraph{Pairwise Influence Matrix.}
Given distinct $u,v\in V$, the influence of $v$ on $u$ is defined as 
$$ I_{v\to u}=\max_{\sigma, \sigma'\in X(d-1)} \norm{\mu_\sigma - \mu_{\sigma'}}_{TV}$$
where the maximum is over all pairs $\sigma,\sigma'$ of choices of spins for all sites in $V \setminus \{u\}$ where the choice of spin is the same for every site in $V \setminus \{u,v\}$. In other words, $I_{v\to u}$ measures how much $v$ can change the marginal distribution of $u$ (in the worst case) when the spins of all other sites are fixed. The {\bf pairwise influence} matrix $I\in \R^{V\times V}$ is defined to be the (not necessarily symmetric) matrix where $I(u,v)=I_{v\to u}$. %The Dobrushin's like condition say that if the max-eigenvalue of $I$ is bounded away from 1, then the Glauber dynamics mixes rapidly.

\paragraph{Spectral Independence.}
Define $\Psi\in \R^{(\cup_u S_u)\times (\cup_u S_u)}$ as follows. Given distinct sites $u,v\in V$, define
  $$\Psi(us_u,vs_v) =\P_{\sigma\sim\mu}[vs_v\in \sigma | us_u\in\sigma] - \P[vs_v\in \sigma],$$ 
  and let $\Psi(us,us')=0$ otherwise.
We say $\mu$ is {\bf $\eta$-spectrally independent} if for any $\tau\in X$ of codimension at  least 2, we have $$\lambda_{\max}(\Psi_{\mu_\tau})\leq \eta. $$
Spectral independence is recently introduced in \cite{ALO24,CGSV21,FGYZ21} and it was shown that it implies that the Glauber dynamics mixes in $O(d^{\eta+1})$-steps. 

The modern form of Dobrushin's condition is then a bound on the maximum eigenvalue of the pairwise influence matrix. This condition implies spectral independence of the distribution $\mu$ on the facets of $X$, which in turn implies rapid mixing of the Glauber dynamics. We formalize this in the following.

\begin{theorem}[see e.g., \cite{Hay06,DGJ09,BCCPSV06,Liu21}]\label{thm:dobrushin}
    If $\rho(I)\leq 1-\eps$, then $\mu$ is $2/\eps$-spectrally independent and the Glauber dynamics mixes in $O(d\log d/\eps)$ steps, where $\rho(I)$ is the spectral radius of $I$ (that is, the largest eigenvalue of $I$ in absolute value).
\end{theorem}
The above theorem has a very long history. Originally Dobrushin \cite{Dob70} showed that for infinitely sized multi-state spin systems, if the influence of every site $u$ on all other sites is bounded away from 1 then the Gibbs distribution is unique. This condition was later weakened (see e.g., \cite{DS85,Wei05}). After the introduction of the path coupling technique, it was  observed that the Dobrushin's condition also implies rapid mixing of the Glauber dynamics. The condition was weakened in a sequence of works \cite{Hay06, DGJ06,DGJ09,BCCPSV06,Liu21} leading to the above theorem.
We remark that \cite{BCCPSV06} only justifies spectral independence of $\mu$.

We also emphasize that Dobrushin's condition has found many applications outside of the theory of the Markov chain community (see e.g., \cite{DOR16,LY25}) because of its simplicity and generality. Although it may not give a tight mixing result, in many settings it provides a close to optimal bound.

\subsection{Main Results}
Given a complex $(X,\mu)$, the (weighted) {\bf skeleton}\footnote{Some authors use the term base graph instead of skeleton.} of $X$, denoted $G_\varnothing$, is the weighted graph with vertex set $X(1)$ and edge set $X(2)$ where the weight of an edge $$w(x,y)=\P_{\sigma\sim\mu}[x,y\in\sigma].$$
We let $A_\varnothing$ be the adjacency matrix of this graph, and $P_\varnothing$ be the transition probability matrix of the simple random walk. More generally, for any face $\tau\in X$ with $\codim(\tau)\geq 2$, we let $A_\tau, P_\tau$ be the adjacency matrix and the random walk matrix of the skeleton of link of $\tau$. 
%Note that for any $\tau \in X$, $\mu$ induces a probability distribution, $\mu_{|\tau}$ on the facets of $X_\tau$ where the probability of a facet $\sigma'$ is $\P_{\sigma\sim\mu}[\sigma'\in\sigma | \tau\in \sigma]$. 
Lastly, we say $X$ is a {\bf connected} complex if for every $\tau$ of codimension at least 2, the skeleton of $X_\tau$ is a connected graph.

Given a $d$-partite complex $X$, we define the {\bf pairwise spectral influence} matrix ${\cal I}$ as follows. For any distinct $u,v\in V$ we define 
$$ \cI(u,v)=\cI(v,u)=\max_{\tau\in X(d-2), u,v\notin V(\tau)}\lambda_2(P_\tau),$$
where $V(\tau)$ is the set of all $v \in V$ with a spin in $\sigma$. 
We note that for any such $\tau$, the graph $G_\tau$ is a bipartite graph; so its second eigenvalue is always non-negative, unless it has exactly two vertices, i.e., both $u,v$ have exactly one spin left in the link of $\tau$. In such a trivial case, we also let $\cI(u,v)=0$ so that $\cI$ always has non-negative entries.

It is straightforward to see the following (see \cref{app:IcI} for a proof):
\begin{lemma}\label{lem:IvscI}
    For any connected $d$-partite complex, $\lambda_{\max}(\cI)=\rho(\cI)\leq \rho(I).$
\end{lemma}
In other words, $I$ measures probabilistic influence of sites on one another (when the spin of every other site is fixed) whereas $\cI$ measures the spectral influence, which can be significantly smaller. 
We remark that spectral independence is closely related to the eigenvalues of links of $X$. Specifically, for any $\tau\in X$ of codimension at least 2, we have $\lambda_2(P_\tau)= \frac{\lambda_{\max}(\Psi_\tau)}{\codim(\tau)-1}$.
See e.g., \cite[Theorem 8 arxiv version]{CGSV21} for a proof.

We show that the Glauber dynamics mixes rapidly even if the largest eigenvalue of the spectral influence matrix is bounded away from 1. 
\begin{theorem}[Main]\label{thm:main}
For any connected $d$-partite complex, if $\lambda_{\max}(\cI)\leq 1-\eps$ then 
for any $\tau\in X$ of codimension at least 2, we have $\lambda_2(P_\tau)\leq \frac{1-\eps}{(\codim(\tau)-1)\eps}$. That is, the distribution $\mu$ is $\frac{1-\eps}{\eps}$-spectrally independent. 
%Furthermore, the spectral gap of the Glauber dynamics is at least $$
\end{theorem}
We remark that the conclusion of the above theorem is slightly weaker than \cref{thm:dobrushin} in the sense that $1/\eps$-spectral independence implies $O(d^{1+1/\eps})$-mixing (unless we make further assumptions about $X$). 
We suspect an improved upper-bound of $O(d^2 \text{polylog}(d) \cdot f(\eps))$ on the mixing time is possible for some function $f(\cdot)$ independent of $d$,
%\textcolor{blue}{(maybe ``mixing time'' here or something? we haven't really mentioned spectral gap yet)} \textcolor{red}{I think we should drop this sentence. I am not sure if it is true in general. }
%is possible
and we leave that as an open question. 

We in fact prove a more general theorem from which \cref{thm:main} follows. To state our theorem we need to equip our set of sites $V$ with a random walk matrix after adding a set of ``boundary'' vertices $B$. Let $G$ be any graph with vertices $V \cup B$ and no self-loops, and let $W\in \R^{(V\cup B)\times (V\cup B)}$ be a random walk matrix, such that from any vertex $v\in V$ there is a path (with positive probability) from $v$ to some vertex in $B$. Given $u,v \in V$, we also denote by $\W_v[u \to v]$ the probability that a random walk starting from $u$ hits $v$ before any boundary vertex in $B$.

%\textcolor{blue}{we should somehow say that the trivial bound is $1$ and not $1 - \frac{1}{d}$ because $|V|=d$ and $d(Q)$ counts the boundary vertex, or else we shouldn't count the boundary vertex in $d(Q)$ I think}

\begin{theorem}[Main technical]\label{thm:maintechnical}
    Let $(X,\mu)$ be a connected $d$-partite complex with parts indexed by $V$, let $W$ be a random walk matrix as described above, and let $M_\varnothing \in \R^{V \times V}$ be the real symmetric matrix with zero diagonal and off-diagonal entries given by $M_\varnothing(u,v) = \sqrt{\W_v[u \to v] \cdot \W_u[v \to u]}$ for all distinct $u,v \in V$. If for all $\sigma \in X(d-2)$ we have $\lambda_2(P_\sigma) \leq \sqrt{W(u,v) \cdot W(v,u)}$, where $u,v$ are the two distinct sites in $V$ which have no spin in $\sigma$, then
    \[
        \lambda_2(P_\varnothing) \leq \frac{\lambda_{\max}(M_\varnothing)}{d-1},
    \]
    and more generally for any $\tau\in X$ of codimension at least 2,
    \[
        \lambda_2(P_\tau)\leq \frac{\lambda_{\max}(M_\tau)}{\codim(\tau)-1},
    \]
    where $M_\tau \in \R^{V \setminus V(\tau) \times V \setminus V(\tau)}$ is defined analogously to $M$, except that the vertices in $V(\tau)$ are considered to be boundary vertices in $B$.
\end{theorem}

To help with the understanding of \cref{thm:maintechnical}, we consider a concrete example. Let $(X,\mu)$ be a $d$-partite complex on sites $V$ labeled by $V = \{1,2,\ldots,d\}$. In the notation of \cite{LLO25}, we call $(X,\mu)$ a \textbf{top-link path complex} if for all $\sigma \in X(d-2)$ we have
\[
    \lambda_2(P_\sigma) \leq \begin{cases}
        \frac{1}{2}, & i = j-1 \\
        0, & i < j-1,
    \end{cases}
\]
where $i < j$ are the two distinct sites in $V$ which have no spin in $\sigma$.

We can capture this definition using our notation as follows. Define $B = \{0,d+1\}$, and define $G$ to be the simple path graph with vertices ordered as $\{0,1,2,\ldots,d,d+1\}$. We let $W$ be the transition probability matrix of the simple random walk on $G$; that is, for each vertex $k \in V$, the walk moves to $k-1$ with $50\%$ probability and to $k+1$ with $50\%$ probability. The $d$-partite complexes $(X,\mu)$ satisfying the condition $\lambda_2(P_\sigma) \leq \sqrt{W(u,v) \cdot W(v,u)}$ for all $\sigma \in X(d-2)$ from \cref{thm:maintechnical} are then precisely top-link path complexes.

% ...............................

% To help with the understanding of \cref{thm:maintechnical}, we consider a concrete example. Let $V \cup B$ be a path of length $d+2$ with vertices labeled $0,1,2,\ldots,d+1$. We let $V = \{1,2,\ldots,d\}$ and $B = \{0,d+1\}$, and we define the random walk matrix $W$ by enforcing that for each vertex $k \in V$, the walk moves to $k-1$ with $50\%$ probability and to $k+1$ with $50\%$ probability. The $d$-partite complexes $(X,\mu)$ satisfying the condition $\lambda_2(P_\sigma) \leq \sqrt{W(u,v) \cdot W(v,u)}$ for all $\sigma \in X(d-2)$ from \cref{thm:maintechnical} are then precisely the \textbf{top-link path complexes} studied in \cite{LLO25}.

In this case, the values of $\W_v[u \to v]$ can be explicitly computed: for $i < j$ we have
\[
    \W_j[i \to j] = \frac{i}{j} \qquad \text{and} \qquad \W_i[j \to i] = \frac{d+1-j}{d+1-i}.
\]
One can then use this with \cref{thm:maintechnical} to prove that $\lambda_2(P_\tau) \leq \frac{\lambda_{\max}(M_\tau)}{d-1}=\frac{1}{2}$ for all $\tau \in X$, which is one of the main results of \cite{LLO25}. We give a more detailed proof of this in \cref{sec:path-complexes}.

We find it quite surprising that hitting probabilities of the given random walk appear in the second eigenvalue bounds of \cref{thm:maintechnical}. Intuitively, we believe that they represent global correlations between sites in the complex.

To apply the theorem more generally, one can consider the following remarks:
\begin{enumerate}
    \item Since the theorem allows any absorbing random walk matrix $W$, one should design $G$ and $W$ such that for any $u\in V$, a random walk starting at $u$ hits a vertex in $B$ as fast as possible.
    % The theorem gives the freedom to choose the random walk matrix $W$, and any choice of $W$ gives a bound on $\lambda_2(P_\tau)$ for all $\tau \in X$. Thus to get the best possible bound on $\lambda_2(P_\tau)$ one needs to design the random walk matrix $W$ together with a carefully chosen set of boundary vertices $B$ such that: for any $u\in V$, a random walk starting at $u$ hits a vertex in $B$ as fast as possible.
    % %without visiting too many (distinct) vertices.
    % This will ensure the entries of $M_\tau$ are small, implying a good upper bound on $\lambda_{\max}(M_\tau)$.
    % % Of course the speed with which the random walk absorbs into $B$ needs to be balanced with the condition $\lambda_2(P_\sigma) \leq \sqrt{P(u,v) \cdot P(v,u)}$.

    % \item Once $W$ has been designed... Morally, the theorem says that if the second eigenvalue of the top (codimension 2) links of $X$ are bounded by the local probabilities of the random walk associated to $W$, then $\lambda_2(X_\varnothing)$ is bounded by a certain global statistic of the random walk associated to $W$.

    % \item For any $u$ and any $Q\sim \cW_{V(\tau)}[u]$, we always have $d(Q)-2\leq \codim(\tau)-1$. This is because other than $u$ and the last vertex of the walk, there are at most $\codim(\tau)-1$ distinct vertices on $Q$. Therefore we always have the trivial bound $\frac{\E_{Q\sim \cW_{V(\tau)}[u]}[d(Q)-2]}{\codim(\tau)-1}\leq 1$. 

    \item For any choice of $W$, we have $M_\tau(u,v) \leq 1$ for all $\tau \in X$ since $\W_u[v \to u]$ is a probability. Therefore a row sum bound on $M_\tau$ gives the trivial bound $\lambda_2(P_\tau) \leq 1$ for all $\tau \in X$.
    % For any choice of $W$, we have $M_\tau(u,v) \leq 1$ for all $\tau \in X$ since $\W_u[v \to u]$ is a probability. Thus the maximum row sum of $M_\tau$ is bounded above by $\codim(\tau)-1$ since $M_\tau$ has zero diagonal. Therefore we always have the trivial bound $\frac{\lambda_{\max}(M_\tau)}{\codim(\tau)-1} \leq 1$ for all $\tau \in X$.

    \item The Cauchy interlacing theorem and \cref{lem:monotoneeigenvalues} imply $\lambda_{\max}(M_\tau) \leq \lambda_{\max}(M_\varnothing)$ for all $\tau \in X$. Thus a constant bound on $\lambda_{\max}(M_\varnothing)$ is sufficient to imply spectral independence.
    % When considering non-empty $\tau \in X$ of codimension at least 2, we are changing the random walk by forcing some vertices of $V$ to be boundary vertices in $B$. Thus the probabilities $\W_v[u \to v]$ can only decrease in value. Therefore $M_\tau$ is entrywise smaller than the corresponding principal submatrix of $M_\varnothing$, which implies $\lambda_{\max}(M_\tau) \leq \lambda_{\max}(M_\varnothing)$ by the Cauchy interlacing theorem and \cref{lem:monotoneeigenvalues}.
\end{enumerate}

We finally give a corollary which gives a slightly weaker but more conceptual form of the upper bound on $\lambda_2(P_\varnothing)$ in \cref{thm:maintechnical}. We will use this corollary to prove \cref{thm:main}.

\begin{corollary}\label{cor:maintechnical-distinct}
    Let $(X,\mu)$ be a connected $d$-partite complex with parts indexed by $V$, and let $W$ be a random walk matrix as described above. If for all $\sigma \in X(d-2)$ we have $\lambda_2(P_\sigma) \leq \sqrt{W(u,v) \cdot W(v,u)}$ where $u,v$ are the two distinct sites in $V$ which have no spin in $\sigma$, then
    \[
        \lambda_2(P_\varnothing) \leq \max_{u \in V} \frac{\E_{Q \sim \cW[u]}[\mathrm{distinct}(Q)-2]}{d-1},
    \]
    where the expectation is over all random walks $Q$ starting at $u$ and stopping the first time it hits a vertex in $B$, and $\mathrm{distinct}(Q)$ is the number of distinct vertices in $Q$. An analogous bound also holds for $\lambda_2(P_\tau)$ for all $\tau \in X$.
    % % For any connected $d$-partite complex $(X,\mu)$, a set of boundary vertices $B$ with random walk matrix $P$ such that for any $u,v\in V$, $\lambda_2(P_\sigma)\leq \sqrt{P(u,v)P(v,u)}$ for any link of $\sigma$ codimension that doesn't have a spin for $u,v$. Then,
    % $$\lambda_2(P_\varnothing)\leq \max_{u \in V} \frac{\E_{Q\sim \cW[u]}[d(Q)-2]}{d-1},$$
    % and more generally for any $\tau\in X$ of codimension at least 2,
    % %\textcolor{blue}{(we could also write
    % %
    % \[
    %     \lambda_2(P_\tau)\leq \max_{u\in V\setminus V(\tau)} \frac{\E_{Q\sim \cW_{V(\tau)}[u]}[d(Q)-2]}{\codim(\tau)-1},
    % \]
    % %but probably too complicated?)
    % %}
    % where
    % % $W[u]$ is a random walk starting at $u$ and stopping the first time it hits a vertex in $B$, and
    % $\cW_{V(\tau)}[u]$ is the random walk starting at $u$ and stopping the first time it hits a vertex in $B\cup V(\tau)$ (with $\cW[u] = \cW_\varnothing[u]$), and $d(Q)$ is the number of distinct vertices in $Q$.

    We note that the RHS of the bound has another interpretation: it is the maximum (over $u \in V$) expected fraction of vertices of $V \setminus \{u\}$ visited before hitting a vertex of $B$ on a random walk starting from $u$.

   % \textcolor{blue}{Or maybe the following version. I think this is exactly right? Maybe with some better notation though.} \textcolor{red}{i dont see much differences between the two versions. you may also write both; say the ration is the fraction of visited vertices.} \textcolor{blue}{I don't think there is a difference really, except to me in terms of intuition. I think the above version emphasizes more the $d-1$ in the denominator, which sort of gives the connotation of spectral independence, whereas the below version feels more conceptual to me somehow: the number itself is a percentage, perhaps emphasizing point (1) written above. Anyway, I think that writing the above expression, and then writing in words the percentage thing is what I'll go with. See above.}
    % \[
    %     \lambda_2(P_\varnothing) \leq \max_{u \in V} \mathrm{V\%}(u),
    % \]
    % where $\mathrm{V\%}(u)$ denotes the expected proportion of $V \setminus \{u\}$ that is visited before hitting a vertex of $B$ on a random walk starting from $u$. An analogous bound also holds for $\lambda_2(P_\tau)$ for all $\tau \in X$.
\end{corollary}

We now apply \cref{cor:maintechnical-distinct} to the path graph example discussed after \cref{thm:maintechnical}. When $d$ is odd (the even case is similar), the worst case choice of $u \in V$ for the bound of \cref{cor:maintechnical-distinct} is $u = \frac{d+1}{2}$, and we want to compute the expected number of distinct vertices visited by a random walk starting at $u$ and ending at either $0$ or $d+1$. By symmetry we may assume that the walk ends at $0$, and thus it is equivalent to ask for the expected largest index of $V$ reached by the walk before hitting $0$. That is, letting $Z$ be the random variable indicating the largest index of $V$ we visit, we want to compute $\E_{Q \sim \cW[u]}[Z \mid \text{absorb at } 0]$.

Obviously $Z \geq \frac{d+1}{2}$, and straightforward computations imply $\P[Z \geq i \mid \text{absorb at } 0] = \frac{d+1-i}{i}$. Thus we have
\[
    \E_{Q \sim \cW[u]}[Z \mid \text{absorb at } 0] = \frac{d+1}{2} + \sum_{i=\frac{d+1}{2} + 1}^d \P[Z \geq i \mid \text{absorb at } 0] \approx d \cdot \ln(2).
\]
Applying \cref{cor:maintechnical-distinct}, we obtain $\lambda_2(P_\varnothing) \leq \ln(2) \approx 0.693$. Note that this is worse than what we can obtain from directly using \cref{thm:maintechnical}, which is $\lambda_2(P_\varnothing) \leq \frac{1}{2}$.

% We make one final note regarding the bounds of \cref{cor:maintechnical-distinct}. For any $u$ and any $Q\sim \cW_{V(\tau)}[u]$, we always have $d(Q)-2\leq \codim(\tau)-1$. This is because other than $u$ and the last vertex of the walk, there are at most $\codim(\tau)-1$ distinct vertices on $Q$. Therefore we still always have the trivial bound $\frac{\E_{Q\sim \cW_{V(\tau)}[u]}[d(Q)-2]}{\codim(\tau)-1}\leq 1$.

\subsection{Applications}
\paragraph{Analysis of Markov Chains.}
The pairwise influence matrix of Dobrushin has
long been used as an analytical tool to bound the
rate of convergence of Gibbs sampling. This is mainly due to the simplicity and locality of the  Dobrushin's criterion. 

We expect our main \cref{thm:main} to find  applications in many settings that the classical influence matrix fall short such as  analyzing Gibbs distributions in statistical physics 
or Graphical models in machine learning.

Given our stronger theorem one wonders for what regime of parameters or for what family multi-state spin systems our new theorem significantly beats the classical results:
Unfortunately for some of the classical examples of multi-state spin systems such as the graph coloring problem the bound of \cref{thm:main} is equivalent to that of the (path) coupling and Dubrushin's bound. It turns out that \cref{thm:main} gives significantly improved bound  for instances with "asymmetric" interactions between the spins. Next, we explain one such example. 

\paragraph{The Multi-state Hard Core Model.}
%The classical \textbf{hard core model} from statistical physics concerns a certain probability distribution on the set of independent sets of a graph $G$. 
Given a graph $G= (V,E)$, and a parameter $\lambda>0$, in the hard core model, the probability of each independent set $I$ of $G$ is proportional to  $\lambda^{|I|}$,
Studying the hardcore model has been pivotal in helping to understand the relationship between phase
transitions in statistical physics and phase transitions in efficient approximability, see e.g., \cite{ALO24,CLV21} for recent applications of the spectral independence machinery in studying the hardcore model. 

The \textbf{multi-state hard core model} is  a multi-state spin system generalization of  the hard core model \cite{Lou91, MS91}. Given an additional parameter $C \in \Z_{>0}$, every vertex $v \in V$ is now associated with the set of spins $S_v = \{0,1,\ldots,C\}$, and a configuration of spins $\{s_v\}_{v \in V}$ is admissible if for every edge $\{u,v\}\in E$ we have $s_u+s_v \leq C$. The probability of such a configuration is proportional to $\lambda^{\sum_v s_v}$. Observe that when $C=1$ this model coincides with the hard core model, where the vertices with spin $s_v = 1$ give an independent set of $G$. The model was first introduced in \cite{Lou91} to study telecommunication networks. Later, it was observed by Mazel and Suhov \cite{MS91} that the model can be treated as a "subensemble" of the well-known Solid-on-Solid (SOS) model in statistical physics. Here we omit the details and refer interested readers to \cite{CLMST14} for background information on the SOS model and connections to statistical physics and sampling. 

Galvin, Martinelli, Ramanan and Tetali \cite{GMRT11} observed that the model undergoes a phase transition at  $\lambda\approx \left(\frac{\log \Delta}{(C+2)\Delta}\right)^{2/(C+2)}$ for even $C$ and at $\lambda \approx \left(\frac{e}{\Delta}\right)^{2/(C+1)}$ for odd $C$, when $G$ is a $\Delta$-ary tree. To the best of our knowledge, despite significant applications, no prior counting and sampling results are known for the multi-state hard core model. 
Here we use \cref{thm:main} to efficiently sample a random configuration of the multi-state hard core model for $\lambda$ up to almost the uniqueness threshold.
%as long as $\lambda \leq \left(\frac{1}{\Delta}\right)^{2/(C+1)}$ (up to lower order terms), where $\Delta$ is the max degree of the underlying graph $G$.
\begin{theorem}
Given a instance of the multi-state hard core model, i.e., a graph $G=(V,E)$, $0\leq \lambda<1$ and $C\in \mathbb{Z}_{>0}$, then  for any $u,v\in V$,
$$ \cI(u,v)\leq \frac{\lambda^{\frac{C+1}{2}}}{1-\sqrt{\lambda}}.$$
Therefore, if $\lambda_{\max}(\cI)\leq 1-\eps$, $\mu$ is $\frac{1-\eps}{\eps}$-spectrally independent and the Glauber dynamics mixes rapidly.
\end{theorem}
A few remarks are in order. If $\Delta$ is the maximum degree of $G$, then
\begin{itemize}
    \item It follows that the Glauber dynamics mixes rapidly as long as 
    $$\lambda\leq \left(\frac{1-\eps}{(C+1)\Delta}\right)^{2/(C+1)}.$$ Note that this converges to $1$ as $C\to\infty$ (and $G$ is fixed). This almost matches the phase transition for even $C$ as $C\to\infty$.
    \item If we fix $\epsilon$ and $C$, then for $\Delta \geq \left(\frac{1}{\epsilon}\right)^{1/(C+1)}$ the Glauber dynamics mixes rapidly as long as
    \[
        \lambda \leq \left(\frac{1-2\epsilon}{\Delta}\right)^{2/(C+1)}.
    \]
    This almost matches the phase transition for odd $C$ as $\Delta \to \infty$.
    \item The (classical) influence of $v$ on $u$ can be as large as $\lambda$ (in the worst case).  So, the classical Dobrushin's/coupling type arguments one can only prove fast mixing if $\lambda<\frac{1-\eps}{\Delta}$ (independent of $C$).
\end{itemize}
\begin{proof}
%We first compute the entries of the pairwise spectral influence matrix $\cI$. 
Fix two sites $u,v \in V$, and let $\tau \in X(d-2)$ be any allowed choice of spins for all other sites except $u,v$. Then based on the spin of sites adjacent to $u,v$, there exist $C_u,C_v \leq C$ such that $\tau \cup \{s_u,s_v\} \in X(d)$ if and only if $s_u \leq C_u$, $s_v \leq C_v$, and $s_u + s_v \leq C$. With this, the probability transition matrix for the remaining vertices $u,v$ is $P_\tau = D^{-1}\left[\begin{smallmatrix} 0 & A \\ A^\top & 0 \end{smallmatrix}\right]$, where $A$ is the $C_u \times C_v$ matrix given by
\[
    a_{s_us_v} = \begin{cases}
        \lambda^{s_u+s_v}, & s_u+s_v \leq C, \\
        0, & \text{otherwise}
    \end{cases}
\]
and $D$ is the diagonal matrix consisting of the row sums of $A$ and $A^\top$.
%We now bound the maximum eigenvalue of $\cI$. 
By \cref{lem:multi-state-codim2}, we have $\lambda_2(P_\tau) \leq \lambda^{(C+1)/2} / (1-\sqrt{\lambda})$ as desired. 
%Based on our assumption on $\lambda$, say $\lambda \leq \left(\frac{1}{\Delta}\right)^{2(1+\log^{-1} \Delta)/(C+1)}$, we further have
%\[
 %   \lambda_2(P_\tau) \leq \frac{\Delta^{-(1+\log^{-1} \Delta)}}{1-\Delta^{-(1+\log^{-1} \Delta)/(C+1)}} = \frac{1}{\Delta} \cdot \frac{1/e}{1-(e\Delta)^{-1/(C+1)}} \leq .
%\]
%Thus the maximum absolute row sum of $\cI$ is at most $...$.
\end{proof}
% \paragraph{Example.} Consider the following concrete example: We are given a family of $d$ individuals (sites) on a graph $G=(V,E)$ where each edge represents mutual friendship. Suppose a spin represents the activity that an individual is doing (during the day) such as working, sleeping etc. A new contagious disease is started in the community; when an individual gets sick, we mark them as ``sick'' and we will immediately send their friends to ``quarantine''. The question is to sample a uniformly random state of such a system. That is for a $\sigma\in X$ let
% $$ \mu(\sigma)\propto \prod_{u\sim v} \I[s_u\neq \text{``sick''} \vee s_v=\text{``quarantine''}] $$
% The following lemma is immediate:
% \begin{lemma}
%     Suppose $\Delta$ is the maximum degree of any individual in $G$ and suppose $|S_v|\geq q$ for every site $v$. There is a universal constant $c>0$ such that if $\Delta<c\sqrt{q}$, then $\lambda_{\max}(\cI)\leq 1/2$ and the Glauber dynamcis mixes in polynomial time.
% \end{lemma}
% We emphasize that for any $u\sim v$ the influence of $v$ to $u$ is almost 1, $I_{v\to u}\geq 1-O(1/q).$ So, $\lambda_{\max}(I)$ is not bounded and the classical (path) coupling arguments/Dobrushin's like conditions fails to analyze the Glauber dynamics in such settings.

\paragraph{Trickle-down Theorems.}
Trickle-down theorems are a family of local-to-global theorems for simplicial complexes. Roughly speaking these theorems imply that if $\lambda_2(P_\tau)$ is small for ``all'' faces of codimension $2$, then the underlying complex is a local spectral expander. 
\begin{theorem}[ \cite{Zuk03,BHV07,Opp18}]\label{thm:oppenheim}
If $(X,\mu)$ is a connected $d$-dimensional such that $\lambda_2(P_\tau)\leq \frac{1-\eps}{d}$ for every $\tau\in X$ of codimension 2,  then for any $\sigma$ with $\codim(\sigma)=k$, $\lambda_2(P_\sigma)\leq \frac{1-\eps}{d-(k-2)(1-\eps)}$. In particular, $X$ is a $\frac{1-\eps}{\eps d}$-local spectral expander.
\end{theorem}
If $X$ is $d$-partite, then the assumption of the above theorem translates to every entry of the pairwise spectral influence matrix is bounded by $\frac{1-\eps}{d}$, so the max eigenvalue is trivially upper-bounded by $1-\eps$. In a sequence of recent works the above theorem has been generalized for partite or path complexes \cite{AO23,LLO25}. Our main \cref{thm:main} can be seen as a generalization of all of these recent works. In particular, see \cref{sec:path-complexes} for a short reproof of spectral expansion of the links of path complexes. %\textcolor{blue}{keep? I think this is fine here}

Lastly, we remark that the trickle-down theorems are often employed in constructions of sparse high dimensional expanders e.g., \cite{LSV05, KO18,OP22}. Our main theorem can also be exploited in such scenarios to obtain improved bounds on the local spectral expansion of the underlying complex.

\subsection{Technical Overview}
Trickle-down theorems can be seen as one of the strongest tools to prove mixing time of Markov chains. In several cases such as sampling bases of matroids \cite{ALOV24} or edge coloring of graph \cite{ALO22,WZZ24} we are not aware of any other technique that can replicate the result. In particular, none of correlation decay, coupling, canonical path, interpolation method, etc., seems to work.

This paper builds on a long-term program \cite{ALO22,WZZ24,AO23,LLO25} where the goal is to obtain near optimal trickle-down theorems for various family of simplicial complexes. To this date, this program has lead to near optimal bounds for sampling edge coloring of graphs \cite{ALO22,WZZ24}, near tight bound on spectral expansion of sparse coset complexes \cite{AO23,LLO25}, and new tight bounds for sampling and log-concavity of path complexes such as distributions over maximal chains in distributive and modular lattices. 
Some of the major goals of this program is to obtain optimal mixing results for sampling proper colorings of graphs, or to prove long-standing conjectures on log-concave sequences corresponding to matroids of lattices such as the Neggers-Stanley conjecture.

Originally, \cite{ALO22,WZZ24,AO23} managed to prove improved trickle-down theorems by a purely linear algebraic argument extending the proof of \cref{thm:oppenheim} by running a more sophisticated inductive proof via a technique called ``matrix trickle-down theorem''. Although they succeeded in giving near optimal mixing time arguments for sampling edge coloring of graphs \cite{ALO22,WZZ24} the technique seemed rather complicated and non-trivial to generalize. It is also typically lossy. In this work, building on \cite{LLO25}, we continue building a new approach that is proving tight trickle-down theorems using the machinery of $\cC$-Lorentzian polynomials. For experts, the polynomial techniques so far are analogous to the diagonal induction machinery introduced in \cite{ALO22}.

\paragraph{$\cC$-Lorentzian polynomials and commutative maps.} In the $\cC$-Lorentzian polynomial technique, we recursively define a family of  polynomials $p_{\sigma}$ over faces $\sigma$ of a complex $X$ such that their Hessian has only one positive eigenvalue when evaluated in certain open convex cones (the definition is in fact more general, see \cref{def:C-Lorentzian}). These polynomials are constructed in such a way that eigenvalues of $P_\sigma$ can be computed by evaluating and taking directional derivatives of $p_\sigma$.
Prior to our works, there has been several suggestive definitions for stable or Lorentzian polynomials with respect to cones \cite{DGT21,AASV21,BD24}. These definitions are incomparable to our machinery since the cone is often fixed; in our machinery the cone changes and gets ``bigger'' as we take more (directional) derivatives. 
For example, we relate the cone of $p_\sigma$ to that of $p_{\sigma\cup \{x\}}$ (for some $x\in X_\sigma(1)$)
by defining a family of ``commutative'' $\pi$ maps
$$ \pi_{\sigma+x}: \R^{X_\sigma(1)}\to \R^{X_{\sigma\cup\{x\}}(1)} \quad \text{for all} \quad \sigma \in X, ~ x \in X_\sigma(1),$$
so that applying $\pi_{\sigma+x}$ to the cone of $p_\sigma$ gives the cone of $p_{\sigma \cup \{x\}}$, where ``commutative'' roughly speaking means
$$ \pi_{\sigma\cup\{x\}+y}(\pi_{\sigma+x}(\cdot)) = \pi_{\sigma\cup\{y\}+x}(\pi_{\sigma+y}(\cdot)) \quad \text{for all} \quad \{x,y\}\in X_\sigma(2).$$
Commutative $\pi$ maps of a specific form are then utilized in the definition of the polynomials $p_\sigma$. Concretely, we define $p_\sigma = \mu(\sigma)$ when $\sigma$ is a facet of $X$, and otherwise we inductively define
\[
    (d-|\sigma|) \cdot p_\sigma(\bm{t}) = \sum_{x \in X_\sigma(1)} t_x \cdot p_{\sigma \cup \{x\}}(\pi_{\sigma+x}(\bm{t})) \quad \text{with} \quad \pi_{\sigma+x}(\bm{t})_y = t_y - \phi_\sigma(y,x) \cdot t_x,
\]
for a carefully chosen $\phi_\sigma(\cdot,\cdot)$.
As we will demonstrate, such families of commutative $\pi$ maps
very generally give rise to trickle-down-type theorems, yielding eigenvalue bounds on all links from eigenvalue bounds on links of codimension 2. Thus the main challenge is to find the right family of commutative $\pi$ maps. (We note here that the commutativity property is equivalent to a highly over-constrained system of linear equations and thus is a priori difficult to satisfy; see \cref{def:pi-maps}.)

\paragraph{Main contribution.} The main contribution of this paper is to define commutative $\pi$ maps based on a family of random walks on the sites $V$. In fact, once we add a carefully chosen set $B$ of absorbing vertices to $V$, any random walk transition matrix $W\in \R^{(V\cup B)\times (V\cup B)}$  that absorbs in $B$ (in a finite number of steps) gives a family of commutative $\pi$ maps 
by letting $\phi_\sigma(y,x)$ (above) be the probability that a random walk starting from the site of $y$ hits the site of $x$ before the sites of $\sigma$ or the any of the boundary vertices $B$. The commutativity property then simply follows from properties of random walks (see \cref{fact:RWprop}).
%defined by certain hitting probabilities of the underlying random walk.

From this we obtain that any absorbing random walk transition matrix $W$ gives rise to a trickle-down theorem. 
This may seem surprising at first, but the hitting probabilities used to define the commutative $\pi$ maps above are capturing long-range correlations in the underlying probability distribution $\mu$.

Finally, to fully define our polynomials $p_\sigma$ we choose the transition matrix $W$ based on the pairwise spectral influence matrix $\cI$ in such a way that, for $\tau \in X$ of codimension 2, the polynomial $p_\tau$ is a (quadratic) Lorentzian polynomial. This implies eigenvalue bounds on the links of codimension 2, enabling us to apply the trickle-down theorem associated to our commutative $\pi$ maps.

%\paragraph{Proof overview.}

%\textcolor{blue}{Do you think something like this is worth adding somewhere in the technical overview? I feel like people kind of complained in the previous paper about wanting some kind of condensed proof strategy or something ``less technical''.}

%As described above, to prove our main results we utilize a family of polynomials $p_\sigma$ for $\sigma \in X$. 
% is a linear map defined as discussed above based on a given random walk matrix $P$. \textcolor{blue}{(cite \cref{sec:polys-pi-maps} for the definition or maybe put the expression here?)}

With these polynomials defined, the proofs of our main results are detailed in \cref{sec:C-Lor-argument} and \cref{sec:spectral-indep}. They generally follow three main steps:
\begin{enumerate}
    \item For all $\sigma \in X$ such that $\codim(\sigma) = 2$, prove that the Hessian of the quadratic form
    \[
        p_\sigma(\bm{t}) = \sum_{\{x,y\} \in X_\sigma(2)} \mu(\sigma \cup \{x,y\}) \cdot \left(t_x t_y - \frac{\phi_\sigma(y,x)}{2} t_x^2 - \frac{\phi_\sigma(x,y)}{2} t_y^2\right)
    \]
    has at most one positive eigenvalue (i.e. $\lambda_2 \leq 0$). 
    Note that including $\pi$ maps has the effect of adding negative values to the diagonal of $\nabla^2p_\sigma$. Without these values, the multi-affine terms in $p_\sigma$ would need to be supported on bases of a matroid (in order to obtain $\lambda_2 \leq 0$), which would not allow for the more general applications of $\cC$-Lorentzian polynomials given in this paper. % yielding more general second eigenvalue bounds beyond the $\lambda_2 \leq 0$ found in the definition of $\cC$-Lorentzian polynomials (see \cref{def:C-Lorentzian}).

    By \cref{thm:conn+quad=Lor} this implies $p_\sigma$ is $\cC_\sigma$-Lorentzian for all $\sigma \in X$ for certain cones $\cC_\sigma$. That is, from eigenvalue bounds for $\sigma \in X$ with $\codim(\sigma) = 2$, we obtain eigenvalue information for all $\sigma \in X$.

    \item Bound the second eigenvalue of $P_{u,v}$, the bipartite random walk matrix between $S_u, S_v$ (in the link of $\varnothing$) where every edge $(s_u,s_v)$ is weighted by $\P_{\sigma\sim\mu}[us_u,vs_v\in \sigma]$. We do this by applying one well-chosen directional derivative $\nabla_{\bm\alpha(w)}$ for each site $w \in V \setminus \{u,v\}$ (where $\bm\alpha(w) \in \overline{C}_\varnothing$), which has the effect of marginalizing out the site $w$. Because $p_\varnothing$ is $C_\varnothing$-Lorentzian, the Hessian of what remains has at most one positive eigenvalue, which yields the desired bound.

    \item Use bounds on all $P_{u,v}$'s together with some basic facts from spectral graph theory to bound $\lambda_2(P_\varnothing)$ (and $\lambda_2(P_\sigma)$ for all $\sigma \in X$); see \cref{lem:dpartiteeigenvalues}. We remark that our approach differs substantially from the matrix trickle-down arguments of \cite{ALO22} in the sense that we don't directly bound eigenvalues of $P_\varnothing$. Instead, as described above, we first use the theory of $\cC$-Lorentzian polynomials to bound eigenvalues of the induced subgraph between every two parts of the underlying $d$-partite graph, and then we combine these bounds to obtain the desired upper bound on $\lambda_2(P_\varnothing)$.
\end{enumerate}
%\textcolor{red}{Maybe remove the next paragraph, not sure if it is adding anything?}
%Beyond these steps, the most difficult part of proving our main results is as discussed above: coming up with the right family of commutative $\pi$ maps. Utilizing a random walk matrix $P$ on the sites $V$ to define these $\pi$ maps is very far from obvious, and finding this construction was the main overall barrier to applying the above techniques and proof strategy to general partite complexes.

\paragraph{Future directions.} It can be seen that our current machinery is not enough to obtain significantly improved bounds on mixing time of the Glauber dynamics for sampling proper colorings of a graph. This is mainly because in our inductive argument we only use  second eigenvalues of $P_\tau$ for faces $\tau$ of codimension 2, and that is not enough even to obtain optimal mixing results for proper colorings of the complete graph. Going forward, we expect a more general family of commutative $\pi$ maps (which we call ``non-diagonal $\pi$ maps'') can exploit the full spectra of links of codimension 2. We hope that this will allow us to extend the current machinery to give improved mixing bounds for fundamental sampling problems.  

\subsection{Acknowledgements}
We would like to thank Prasad Tetali for pointing us to possible applications of our results to the multi-state hard core model. 
The first author acknowledges the support of the Natural Sciences and Engineering Research Council of Canada (NSERC), [funding reference number RGPIN-2023-03726]. Cette recherche a \'et\'e partiellement financ\'ee par le Conseil de recherches en sciences naturelles et en g\'enie du Canada (CRSNG), [num\'ero de r\'ef\'erence RGPIN-2023-03726].

The second 
author’s research is supported by an NSF grant CCF-2203541, a Simons Investigator Award 928589, and a
Lazowska Endowed Professorship in Computer Science \& Engineering.

\section{Preliminaries}
% \subsection{Notation}
% For $\sigma\in X$, we write $V(\sigma)$ to denote the set of sites that have a spin in $\sigma$ and we write $V_\sigma=V\setminus V(\sigma)$ to denote the rest.

\subsection{Log-concave Polynomials}

We recall the definition of $\cC$-Lorentzian polynomial from \cite{LLO25}, along with the main lemma which makes such polynomials useful for bounding second eigenvalues (and for log-concavity, but we do not utilize this in this paper).

\begin{definition}[$\cC$-Lorentzian polynomials; Defn. 2.1 of \cite{BL23}] \label{def:C-Lorentzian}
    Let $p \in \R[t_1,\ldots,t_n]$ be a $d$-homogeneous polynomial for $d \geq 2$, and let $\cC \subset \R^n$ be a non-empty open convex cone. We say that $p$ is \textbf{$\cC$-Lorentzian} if
    \begin{enumerate}
        \item[(P)] $\nabla_{\bm{v}_1} \cdots \nabla_{\bm{v}_d} p(\bm{t}) > 0$ for all $\bm{v}_1,\ldots,\bm{v}_d \in \cC$, and
        \item[(Q)] the Hessian of $\nabla_{\bm{v}_1} \cdots \nabla_{\bm{v}_{d-2}} p(\bm{t})$ has at most one positive eigenvalue for all $\bm{v}_1,\ldots,\bm{v}_{d-2} \in \cC$.
    \end{enumerate}
    Note that (P) implies any Hessian in (Q) also has at least one positive eigenvalue.
    
    Additionally, any linear form for which (P) holds is defined to be $\cC$-Lorentzian of degree $d=1$, and any positive constant is defined to be $\cC$-Lorentzian of degree $d=0$. %Finally, we consider the zero polynomial to be $\cC$-Lorentzian of every degree $d$.
\end{definition}

% The following theorem is the main tool that we use to prove that the polynomials that we construct by our $\bm{\pi}$-maps are $\cC$-Lorentzian. 
% \begin{theorem}[{\cite[Thm. 2.9]{BL23}}] \label{thm:Lor-from-derivs}
%     Let $p \in \R[t_1,\ldots,t_n]$ be a $d$-homogeneous polynomial for $d \geq 3$, and let $\cC$ be a non-empty open convex cone in $\R_{>0}^n$. If
%     \begin{enumerate}
%         \item $\nabla_{\bm{v}_1} \cdots \nabla_{\bm{v}_d} p > 0$ for all $\bm{v}_1,\ldots,\bm{v}_d \in \cC$, and
%         \item the Hessian of $\nabla_{\bm{v}_1} \cdots \nabla_{\bm{v}_{d-2}} p$ is irreducible and its off-diagonal entries are non-negative for all $\bm{v}_1,\ldots,\bm{v}_{d-2} \in \cC$, and
%         \item $\partial_{t_i} p$ is $\cC$-Lorentzian for all $i$,
%     \end{enumerate}
%     then $p$ is $\cC$-Lorentzian.
% \end{theorem}
% We remark that the original statement is for ``effective'' cones as defined in \cite{BL23}, but note that any cone contained in $\R_{>0}^n$ is automatically effective. We also note that if $p$ in the previous theorem does not depend on some variable $t_i$, then we can apply the theorem to $p$ restricted to all variables except $t_i$.

\begin{lemma} \label{lem:main-purpose-lemma}
    Let $p$ be a $d$-homogeneous polynomial on $\R^n$, and let $\cC \subset \R^n$ be a non-empty open convex cone. If $p$ is $\cC$-Lorentzian and $\bm{u},\bm{w},\bm{v}_1,\ldots,\bm{v}_{d-2} \in \overline{\cC}$ (the closure of $\cC$), then:
    \begin{enumerate}
        \item the Hessian of $\nabla_{\bm{v}_1} \cdots \nabla_{\bm{v}_{d-2}} p$ has at most one positive eigenvalue, and
        \item the coefficients of $f(x) = p(\bm{u} \cdot x + \bm{w})$ form an ultra log-concave sequence.
    \end{enumerate}
\end{lemma}
\begin{proof}
    For (1), let $(\bm{v}_1(i),\ldots,\bm{v}_{d-2}(i))_{i=1}^\infty$ be a sequence of vectors in $\cC$ which approach $(\bm{v}_1,\ldots,\bm{v}_{d-2})$. Thus the Hessian of $\nabla_{\bm{v}_1(i)} \cdots \nabla_{\bm{v}_{d-2}(i)} p$ has at most one positive eigenvalue for all $i \geq 1$. Therefore the limit of these Hessian matrices also has at most one positive eigenvalue.

    The statement of (2) follows from Remark 2.4 of \cite{BL23}.
\end{proof}

\subsection{Linear Algebra} \label{sec:lin-alg}

We next state some eigenvalue bounds which will be useful in our analysis. The first two results are standard, whereas \cref{lem:eigenvaluebipartite} and a weaker version of \cref{lem:dpartiteeigenvalues} were both proven and used in \cite{LLO25}.

\begin{lemma}[\cite{Tro01}]\label{lem:rhonorm}
    For any matrix $A$ and any matrix norm,
\begin{equation}\label{eq:rhodef} \rho(A)=\lim_{n\to \infty} \norm{A^n}^{1/n}
\end{equation}
\end{lemma}
For a square (non-symmetric) matrix $A\in \R^{n\times n}$ let $\bar{A}$ be the symmetrized version of $A$ where $\bar{A}_{i,j}=\sqrt{A_{i,j}A_{j,i}}$.  We will prove the following two lemmas.
The following lemmas are immediate consequences of the above.
\begin{lemma}[{\cite[Example 7.10.2]{Mey00}}]
    \label{lem:monotoneeigenvalues}
    For any pair of matrix $A,B\in \R_{\geq 0}^{n\times n}$ such that $A(i,j)\leq B(i,j)$ for all $i,j$ we have $\rho(A)\leq\rho(B).$ 
\end{lemma}

\begin{lemma}\label{lem:rhoaverage}
    For any $A\in \R_{\geq 0}^{n\times n}$ we have
    $$ \rho(A)\geq \rho(\bar{A})$$
    where $\rho$ is the spectral radius.
\end{lemma}
\begin{proof}
We choose the trace norm $\norm{A} = \sqrt{\textup{Tr}(AA^T)}=\sqrt{\sum_{i,j} A_{i,j}^2}$. 
By \cref{lem:rhonorm}, it is enough to show  for any integer $n\geq 1$,
    $$ \norm{A^n} \geq \norm{\bar{A}^n}.$$
We show that for any $n$ and $i,j$ we have
    \begin{equation}\label{eq:Anij} (A^{n}_{i,j})^2 +(A^n_{j,i})^2 \geq 2(\bar{A}^n_{i,j})^2.
    \end{equation}
    Then, the lemma immediately follows from the fact that $\norm{A^n}^2=\sum_{i,j} (A_{i,j}^n)^2 = \sum_{i,j} (A_{j,i}^n)^2$.

    %\textcolor{red}{perhaps we should double check the $i=j$ case.} \textcolor{blue}{how is it different?}\textcolor{red}{I guess not really any different :)}
    It remains to show \eqref{eq:Anij}. First we notice that $A_{i,j}^n$ can be written as the sum over all walks of length $n$ from $i$ to $j$ (with non-zero diagonal entries of $A$ indicating loops):
    $$ A^n_{i,j} = \sum_{P=(u_0,\dots,u_n), i=u_0,j=u_n} \prod_{i=1}^{n} A_{u_{i-1},u_i} = \sum_{Q=(u_0,\dots,u_n), i=u_0,j=u_n} A_{\rightarrow Q}$$
    where we let $A_{\rightarrow Q}=\prod_{i=1}^{n} A_{u_{i-1},u_i}$. Similarly, 
    % $$ A_{j,i}^n = \sum_Q\prod_{i=1}^n A_{i,i-1}=\sum_Q A_{\leftarrow Q}$$
    $$ A^n_{j,i} = \sum_{P=(u_0,\dots,u_n), i=u_0,j=u_n} \prod_{i=1}^{n} A_{u_i,u_{i-1}} = \sum_{Q=(u_0,\dots,u_n), i=u_0,j=u_n} A_{\leftarrow Q}$$
    where we let $A_{\leftarrow Q}=\prod_{i=1}^{n} A_{u_i,u_{i-1}}$.
    Since $\bar{A}$ is symmetric we have $\bar{A}_{\rightarrow Q}=\bar{A}_{\leftarrow Q}$; furthermore since it is a symmetrization of $A$ for any path $Q$ we have
    $$ \sqrt{A_{\rightarrow Q}A_{\leftarrow Q}} = \bar{A}_{\rightarrow Q}$$
    Therefore, by the Cauchy-Scwhartz inequality we can write
    $$ \left(\bar{A}^n_{i,j}\right)^2 = \left(\sum_{Q=(u_0,\dots,u_n),i=u_0,j=u_n} \sqrt{A_{\rightarrow Q}A_{\leftarrow Q}}\right)^2 \leq \sum_Q A_{\rightarrow Q}\cdot \sum_Q A_{\leftarrow Q} = A_{i,j}^n A_{j,i}^n$$
Now, by AM-GM we have
$$ A_{i,j}^nA_{j,i}^n \leq \frac12 ( (A_{i,j}^n)^2 + (A_{j,i}^n)^2)$$
as desired. 
\end{proof}

\subsection{Spectral Graph Theory}
In this subsection, we state and prove several facts about bipartite and $d$-partite graphs.

For an undirected (possibly weighted) graph $G=(V,E)$, let $P$ be the transition probability of the simple random walk on $G$. Let $\mu$ be the stationary distribution of the walk.
For $f,g:V\to\R$, let 
$$\langle f,g\rangle_\mu:=\E_{x\sim \mu} f(x)g(x)$$
Recall that $P$ is self-adjoint with respect to this inner product, i.e., $\langle Pf,g\rangle_\mu=\langle f,Pg\rangle_\mu.$ Furthermore, the all-ones vector $\bone$ is an eigenfunction with eigenvalue 1.

\begin{lemma}[Eigenvalues of Bipartite graphs]\label{lem:eigenvaluebipartite}
Let $G=(X,Y)$ be a (weighted) bipartite graph, $A\in\R^{(X+Y)\times (X+Y)}$ be the adjacency matrix of $G$, $P$ be the transition probability matrix of the simple random walk on $G$, and $D$ be the diagonal matrix of vertex degrees, i.e., $P=D^{-1}A$. Let $S, \bar{S}\in \R^{(X+Y)\times (X+Y)}$ be diagonal matrices such that $S(x,x)=s_X, S(y,y)=s_Y$ for all $x\in X, y\in Y$ and $\bar{S} = \sqrt{s_Xs_Y} \cdot I_{X+Y}$. Then the following are equivalent:
\begin{enumerate}
    \item $D^{-1/2} A D^{-1/2}\preceq S+vv^T$ for some vector $v\in \R^{X+Y}$,
    \item $D^{-1/2} A D^{-1/2}\preceq \bar{S}+ww^T$ for some vector $w \in \R^{X+Y}$,
    \item $\lambda_2(P) \leq \sqrt{s_X s_Y}$,
    \item $\lambda_2(P-S) \leq 0$.
\end{enumerate}
\end{lemma}

In what follows, given a vector $U \in \R^{X \cup Y}$ we define $u_X$ to be the vector with $X$ entries equal to those of $u$ and $Y$ entries equal to $0$ (and we define $u_Y$ similarly).

\begin{lemma} \label{lem:refined-bipartite-eigs}
    Let $P$ be a random walk matrix on a  bipartite graph $(X,Y)$ with stationary distribution $\mu$, self-adjoint with respect to $\langle \cdot, \cdot \rangle_\mu$. If $g$ is such that $ \langle g,\bone_X\rangle_\mu=\langle g, \bone_Y \rangle_\mu = 0,$ 
    then
    \[
        \langle g, Pg \rangle_\mu \leq \lambda_2(P) \cdot 2 \sqrt{\langle g_X, g_X \rangle_\mu \langle g_Y, g_Y \rangle_\mu}.
    \]
\end{lemma}
\begin{proof}
    First notice that if $g_X={\bf 0}$ then $\langle g,Pg\rangle_\mu=0$ since $G$ is bipartite; so the statement holds trivially. Similarly, we are done if $g_Y={\bf 0}$. So, we assume $g_X,g_Y\neq {\bf 0}$. 
    Let $h=\alpha g_X + \beta g_Y$ for some $\alpha\cdot \beta=1$ that we choose later.
    Since $G$ is bipartite and $\alpha\cdot\beta=1$, $\langle h,Ph \rangle_\mu=\langle g,Pg\rangle_\mu$.
    Since $\langle h,\bone\rangle_\mu=\langle g,\bone\rangle_\mu=0$, we have
    $$
    \langle h,Ph\rangle_\mu\leq \lambda_2(P) \cdot \langle h,h\rangle_\mu = \lambda_2(P)\cdot (\alpha^2\langle g_X,g_X\rangle_\mu + \beta^2\langle g_Y,g_Y\rangle_\mu)
    $$    
 Since $g_X,g_Y\neq {\bf 0}$, $\langle g_X,g_X\rangle_\mu,\langle g_Y,g_Y\rangle_\mu>0$.
    Letting $\alpha^2=\sqrt{\frac{\langle g_Y,g_Y\rangle_\mu}{\langle g_X,g_X\rangle_\mu}}$ and $\beta^2=1/\alpha^2$
    proves the claim.
\end{proof}

% \begin{lemma}\label{lem:dpartiteeigenvalues}
%   Let $G=(V,E,w)$ be a weighted $d$-partite graph with parts $S_1,\dots,S_d$ such that for $u,v\in V$ with $u\neq v$, and any $s_u\in S_u$, we have $w(us_u)=\sum_{s_v\in S_v} w(us_u,vs_v)=\frac{d_w(us_u)}{d-1}$. Here, $d_w(us_u)$ is the weighted degree of $us_u$. Let $P$ be the simple random walk on $G$ with stationary distribution $\nu$.
%   Then, $\nu(S_u)=1/d$ for all $i$.
%    For every $u,v\in V$, let $G_{u,v}$ be the induced bipartite graph on parts $S_u,S_v$ with corresponding random walk matrix $P_{u,v}$. Suppose $P_{u,v} -M_{u,v}$ has at most one positive eigenvalue, where $M_{u,v}$ is the diagonal matrix with entries given by $m_u(v)$ for vertices in $S_u$ and $m_v(u)$ for vertices in $S_v$ such that \textcolor{red}{Do we need both conditions?}$0\leq m_u(v),m_v(u)\leq 1$ and $0\leq \frac12(m_u(v)+m_v(u))< 1$ for all $u\neq v$. Then,
%      $$ \lambda_2(P) \leq 1- \min_{u\in V}\sum_{v:u\neq v} \frac{1-m_u(v)}{d-1} = \max_{u\in V} \sum_{v:u\neq v} \frac{m_u(v)}{d-1}$$
%  \end{lemma}

\begin{lemma} \label{lem:dpartiteeigenvalues}
    Let $G=(V,E,w)$ be a weighted $d$-partite graph with parts $T_1,\dots,T_d$ such that for $i,j\in [d]$ with $i\neq j$, and any $x\in T_i$, we have $w(x,T_j)=\sum_{y\in T_j} w(x,y)=\frac{d_w(x)}{d-1}$. Here, $d_w(x)$ is the weighted degree of $x$. Let $P$ be the simple random walk on $G$ with stationary distribution $\mu$. Then, $\mu(T_i)=1/d$ for all $i$.
    For all distinct $i,j\in [d]$, let $G_{i,j}$ be the induced bipartite graph on parts $T_i,T_j$ with corresponding random walk matrix $P_{i,j}$, and let $M(i,j) \geq 0$ be such that $\lambda_2(P_{i,j}) \leq M(i,j)$. Then
    \[
        \lambda_2(P) \leq \frac{\lambda_{\max}(M)}{d-1},
    \]
    where $M \in \R^{d \times d}$ is the real symmetric matrix with entries given by $M(i,j)$.
\end{lemma}
\begin{proof}
     For  $i\in [d]$, define 
    \[
        e_i(x) =\begin{cases} d-1 & \text{if } x\in T_i\\ -1& \text{otherwise.}\end{cases}
    \]
    Then since for any $i\neq j$ and $x\in T_i$, $P(x,T_j)=\frac{1}{d-1}$ we get that  $e_i$ is also an eigenfunction of $P$ with eigenvalue $\frac{-1}{d-1}$, i.e, $Pe_i=\frac{-e_i}{d-1}$. We call $\bone$ and $e_i$'s trivial eigenfunctions of $P$. 
    
    Let $f$ be a non-trivial eigenfunction of $P$; for any $i\in [d]$  we have 
    $$0=\langle f,\bone+e_i\rangle_\mu= d\cdot \langle f,\bone_i\rangle_\mu.$$ 
    where $\bone_i$ is the indicator vector of the set $T_i$.
    
    Note that if $\lambda_2(P)\leq 0$ we are already done since $M(i,j)\geq 0$ for all $i,j$. So we assume $\lambda_2(P)\geq 0$.
    % It follows by \cref{lem:rayleigh} that the second largest eigenvalue of  $P$ is
    % \[
    %     1-\min_{f:\langle f,\bone_i\rangle_\mu =0,\forall i} \frac{\cE(f,f)}{\E_x f(x)^2}.
    % \]
    Let $g$ be the second eigenfunction of $P$. Consider the bipartite graph $G_{i,j}$ for some $i\neq j$. Let $\mu_{G_{i,j}}$ be the stationary distribution of the random in the graph $G_{i,j}$. Notice that since $d_{G_{i,j}}(x) = \frac{1}{d-1}d_w(x)$ for all $x\in T_i\cup T_j$ we have $\mu_{G_{i,j}}(x) = \frac{d}{2}\mu(x)$ for any $x\in T_i\cup T_j$. Therefore
    %\begin{equation}\label{eq:g1i=g1j}
    \[
        \langle g,\bone_i\rangle_{\mu_{G_{i,j}}} =\langle g,\bone_j\rangle_{\mu_{G_{i,j}}}=0,
    \]
    %\end{equation}
    % Since $P_{i,j}$ is a stochastic, we have that $\bone_{i,j} = (\bone_i,\bone_j)$ is a right eigenvector of $P_{i,j}$ with eigenvalue $1$.
    and thus by \cref{lem:refined-bipartite-eigs}, we have
    \[
        \langle g, P_{i,j} g \rangle_{\mu_{G_{i,j}}} \leq \lambda_2(P_{i,j}) \cdot 2 \sqrt{\langle g_i, g_i \rangle_{\mu_{G_{i,j}}} \langle g_j, g_j \rangle_{\mu_{G_{i,j}}}} \leq M(i,j) \cdot 2 \sqrt{\langle g_i, g_i \rangle_{\mu_{G_{i,j}}} \langle g_j, g_j \rangle_{\mu_{G_{i,j}}}}.
    \]
    We now define
    \[
        w(i) = \sqrt{\langle g_i, g_i \rangle_\mu} = \sqrt{\sum_{x \in T_i} \mu(x) \cdot g(x)^2}
    \]
    for all $i \in [d]$. Since $\mu(x) = \frac{2}{d} \mu_{G_{i,j}}(x)$ for all $x \in T_i$, we then have
    \[
        \langle g, P_{i,j} g \rangle_{\mu_{G_{i,j}}} \leq M(i,j) \cdot d \sqrt{\langle g_i, g_i \rangle_\mu \langle g_j, g_j \rangle_\mu} = M(i,j) \cdot d \cdot w(i) \cdot w(j).
    \]
    % \[
    %     M'(i,j) = M(i,j) \cdot \frac{w(j)}{w(i)}
    % \]
    % for all distinct $i,j \in [d]$. Let $B_{i,j}$ be the diagonal $(T_i \cup T_j) \times (T_i \cup T_j)$ diagonal matrix with $T_i$ entries given by $M'(i,j)$ and $T_j$ entries given by $M'(j,i)$. By \eqref{eq:g1i=g1j} and \cref{lem:refined-bipartite-eigs}, we have
    % \[
    %     \langle g, P_{i,j}g \rangle_{\mu_{G_{i,j}}} \leq \langle g, B_{i,j}g \rangle_{\mu_{G_{i,j}}} = M(i,j) \cdot \left(\frac{w(j)}{w(i)} \langle g_i, g_i \rangle_{\mu_{G_{i,j}}} + \frac{w(i)}{w(j)} \langle g_j, g_j \rangle_{\mu_{G_{i,j}}}\right)
    % \]
    % % Since $P_{i,j}$ self-adjoint with respect to $\langle \cdot, \cdot \rangle_{\mu_{G_{i,j}}}$, it follows from \eqref{eq:g1i=g1j} and \cref{lem:rayleigh} that
    % % \[
    % %     \frac{\langle g, P_{i,j} g \rangle_{\mu_{G_{i,j}}}}{\langle g, g \rangle_{\mu_{G_{i,j}}}} \leq \lambda_2(P_{i,j}) \leq M(i,j).
    % % \]
    % % Note also that
    Using the fact that for $x \in T_i$ and $y \in T_j$ we have $\mu(x) = \frac{2}{d} \mu_{G_{i,j}}(x)$ and $P(x,y) = \frac{P_{i,j}(x,y)}{d-1}$, we compute
    % \begin{align*}
    %     \langle g, Pg \rangle_\mu &= \sum_{i < j} \frac{2}{d} \cdot \frac{1}{d-1} \cdot \langle g, P_{i,j}g \rangle_{\mu_{G_{i,j}}} \\
    %         &\leq \sum_{i < j} \frac{M(i,j)}{d-1} \cdot 2 \cdot w(i) \cdot w(j) \\
    %         % &= \sum_{i < j} \frac{M(i,j)}{d-1} \cdot 2 \cdot w(i) \cdot w(j) \\
    %         &= \sum_{i \neq j} \frac{M(i,j)}{d-1} \cdot w(i) \cdot w(j) \\
    %         &= \frac{\langle w, Mw \rangle}{d-1},
    % \end{align*}
    \[
        \langle g, Pg \rangle_\mu = \sum_{i < j} \frac{2}{d} \cdot \frac{1}{d-1} \cdot \langle g, P_{i,j}g \rangle_{\mu_{G_{i,j}}} \leq 2 \sum_{i < j} \frac{M(i,j)}{d-1} \cdot w(i) \cdot w(j) = \frac{\langle w, Mw \rangle}{d-1},
    \]
    where without loss of generality we assume that $M(i,i) = 0$ for all $i$. We then further have
    \[
        \langle g, g \rangle_\mu = \sum_x \mu(x) \cdot g(x)^2 = \sum_i \sum_{x \in T_i} \mu(x) \cdot g(x)^2 = \sum_i w(i)^2 = \langle w, w \rangle,
    \]
    which implies
    \[
        \frac{\langle g, Pg \rangle_\mu}{\langle g, g \rangle_\mu} \leq \frac{1}{d-1} \cdot \frac{\langle w, Mw \rangle}{\langle w, w \rangle}.
    \]
    Therefore $\lambda_2(P) \leq \frac{\lambda_{\max}(M)}{d-1}$ by the variational characterization of eigenvalues.
\end{proof}

\section{Random Walks on Absorbing Graphs} \label{sec:random-walks-absorbing}
Let $G=(V\cup B, E)$ be a (possibly directed) graph. We call $B$ the set of ``boundary'' vertices. Let  $W\in \R^{(V \cup B)\times (V \cup B)}$ be the transition probability matrix of a simple random walk on $G$. Without loss of generality we assume $W(u,v)>0$ for any (directed) edge $u\sim v\in E$. 
\begin{definition}[Absorbing graph]
    We say $G$ is an absorbing graph if for every vertex $v\in V$ there is a directed path with from $v$ to $B$ (with positive probability).
\end{definition}
Unless otherwise specified, all graphs we work with in this manuscript are assumed to be absorbing.
The following fact is immediate:
\begin{fact}
    If $G$ is an absorbing graph then for any $U\subseteq V$, the graph with boundary vertices $U\cup B$ and vertex set $V\setminus U$  is also absorbing.
\end{fact}
\begin{proof}
    This can be shown by induction. Suppose statement holds for $U$; consider $U'=U\cup\{v\}$. For any vertex $u\in V\setminus U'$, either $U$ has a path to $v$ or a path to $U\cup B$ (by IH).
\end{proof}
Consider a simple random walk on $G$ started at a vertex $v\in V$. Unless otherwise specified, we assume that such a walk will terminate (or be absorbed) when it hits a boundary vertex. This is despite the fact that we may have $W(b,u)>0$ for $b\in B$. 
The following fact is immediate:
\begin{fact}
    For any absorbing graph and any vertex $v\in V$, the simple random walk started at $v$ will be absorbed almost surely in a finite number of steps.
\end{fact}
\begin{proof}
    Let $\eps=\min_{(u,v)\in E}W(u,v)$ be the minimum probability in $P$. Then, for any $v\in V$, the simple random walk from $v$ hits $B$ with probability $\geq \eps^{n-1}$ where $n=|V|.$ So, the probability that the walk is not absorbed by time $t=k(n-1)$ is at most $(1-\eps^{n-1})^{t/(n-1)}$ which $\to 0$ as $t\to\infty$.
\end{proof}
For a walk $Q=(v_i)_{i=0}^\ell$, we use $\W[Q]=\prod_{i=0}^{\ell-1}W(v_i,v_{i+1})$
to denote the probability of going along this walk in our random walk.
For $U\subseteq V$, and a pair of vertices $u,v \in V \cup B$ let 
$$\cW_U[u\to v] = \bigcup_{\ell\geq 0}\{(v_i)_{i=0}^\ell: v_0=u, v_\ell=v \text{ and } v_1,\dots,v_{\ell-1}\notin B\cup U\}$$ 
be the set of all walks from $u$ to $v$ (in $G$)  that do not hit any vertex of $U\cup B$ in between.
Furthermore, let
\[
    \W_U[u\to v]=\sum_{Q\in \cW_{U}[u\to v]} \W[Q]%\prod_{i=0}^{\ell-1}P(v_i,v_{i+1})
\]
be the sum of the probabilities of all  walks in $\cW_U[u\to v]$. When $U = \{w\}$, we may also write $\cW_w[u \to v]$ and $\W_w[u \to v]$.
%, over all walks on $G$ hitting anything in $U$ in between, of the product of the edge weights of the walk. 

Let us explain some concrete examples:
\begin{itemize}
\item If $v \in U$ and $u\neq v$, then $\W_U[u\to v]$ is the {\bf probability} that, on $G$ with boundary vertices $U\cup B$,  a simple random walk started at $u$ is absorbed at $v$. So, in particular, $0\leq \W_U[u\to v]\leq 1$ because $G$ is absorbing.
%\textcolor{blue}{I think we need $u \neq v$ here.}
\item If $v \not\in U$ then the walk is allowed to visit $v$ multiple times, and 
we may have $\W_U[u\to v]>1$. In fact, in the special case that $v=u$ we always have $\W_U(u\to u)\geq 1$.
\end{itemize}
The following fact is crucially used in our construction of $\cC$-Lorentzian polynomials. 
\begin{lemma}\label{fact:RWprop}
    Fix any set $U\subseteq V$ and any three vertices $v\in V\setminus U$ and $u,w\in V$. If $u\neq v$ and $w\neq v$ then we have
\begin{align*}
    \W_U[u \to w] &=\W_{U \cup \{v\}}[u \to w] +\W_{U \cup \{v\}}[u \to v] \cdot \W_{U}[v \to w], %\\
 %   &=\W_{U \cup \{v\}}[u \to w] +\W_U[u \to v] \cdot \W_{U\cup\{v\}}[v \to w]
\end{align*}
%\textcolor{blue}{We at least need $u \neq v$ and $v \neq w$ for this, and we should probably write the other relations here for the $\alpha$ vectors.
and if $v = w$ then we have
\[
    \W_U[u \to v] = \W_{U \cup \{v\}}[u \to v] \cdot \W_{U}[v \to v].
\]
% and if $u = v$ then we have
% \[
%     \W_U[u \to w] = \W_U[u \to u] \cdot \W_{U\cup\{u\}}[u \to w].
% \]
\end{lemma}
\begin{proof}
%Let $W_{a,b}^S$ be the set of walks from $a$ to $b$ that do not hit $S$ in between.
%We prove the first identity; the second one can be proven similarly. 
We start by proving the identity when $u\neq v,w\neq v$.
We give a bijection between walks counted in the LHS and the RHS. By definition, the LHS is the  the sum probabilities of all walks in $\cW_U[u\to w]$. Let $Q\in \cW_U[u\to w]$. We show $Q$ is counted (once) in the RHS. 

{\bf Case 1: $Q\in \cW_{U\cup \{v\}}[u\to w]$.} Then $\W[Q]$ is counted  in $\W_{U\cup \{v\}}[u\to w]$. 

{\bf Case 2: $Q\notin W_{U\cup\{v\}}[u\to w]$.} In other words, $v$ is hit (not at the end) in $Q$. Consider the first time that it hits $v$. We can write $Q=Q_1Q_2$, i.e., we write $Q$ as the concatenation of two walks where
$$Q_1\in \cW_{U\cup \{v\}}[u\to v],\quad\quad Q_2\in \cW_U[v\to w],\quad \quad \W[Q]=\W[Q_1] \cdot \W[Q_2]. $$
Note that $Q_1$ is counted in $\W_{U\cup \{v\}}[u\to v]$ and $Q_2$ is counted in $\cW_{U}[v\to w].$

Conversely, we show that any walk in the RHS is counted (once) in the LHS. First, any $Q_1\in \cW_{U\cup \{v\}}[u\to v]$ and $Q_2\in Q_U[v\to w]$ can be concatenated to form $Q_1Q_2\in \cW_U[u\to w]$ which hits $v$ in the middle. $\cW_{U\cup \{v\}}[u\to w]$ makes up the rest of walks in $\cW_U[u\to w].$ 
\end{proof}

\begin{lemma}\label{lem:minwuv}
    For any pair of vertices $u,v\in V$ with $u\neq v$ we have
    $$ 0\leq \W_v[u\to v]\leq 1, \quad\quad \min\{\W_v[u\to v], \W_u[v\to u]\}<1.$$
\end{lemma}
\begin{proof}
    First observe that for any $u,v\in V$, $\W_v[u\to v]$ is the probability that a simple random walk started at $u$ hits $v$ before any boundary vertex in $B$. 
    This proves the first assertion of the lemma.
    To prove the second assertion, we use that $G$ is absorbing. 
    Therefore, there is a directed path from each $u$ and $v$ to $B$. Let $Q$ be the {\bf shortest} path from any of $u$ or $v$ to $B$; say it is from $u$. But then, this path does not visit $v$ (otherwise there is an even shorter path). This implies, that the walk started at $u$ has a positive probability of hitting $B$ before hitting $v$; i.e., $\W_v[u\to v]<1$ as desired. 
\end{proof}
%\subsection{Reversible Graphs}
%Let $G=(V\cup B,E)$ be a (possibly directed) graph with a transition probability matrix $P$ (supported on $E$). Let $\nu\in \R^{V\cup B}_{>0}$ be a stationary distribution of $P$, i.e., $\nu^T P =\nu^T$ and $\nu$ is a left eigenvector of $P$ with eigenvalue $1$ such that $\nu(v)>0$ for all $v \in V \cup B$. Note that here we are not assuming $G$ is (strongly) connected so it may have multiple stationary distributions, but we consider one which is non-zero everywhere.
%We say $P$ is a {\em reversible} random walk if for all $u,v\in V \cup B$
%$$ \nu(u)P(u,v)=\nu(v)P(v,u).$$
%It is well-known that reversible random walks correspond to random walks on {\bf undirected} graphs.
%Given a reversible random walk on an (undirected) graph $G = (V \cup B, E)$, w
We now define $\cW[u]$
to be the set of walks which start at $u\in V$, and end the first time a vertex in $B$ is reached. %\textcolor{blue}{should this stop immediately if $u \in U$?}\textcolor{red}{We don't really need it. we just need $W[u]$. So, we may only define $W[u]$ what do you think?} \textcolor{blue}{Yeah maybe that's better actually, just $W[u]$.}
Since we assume $G$ is absorbing, for any $u\in V$ we have
$$ \sum_{Q\in \cW[u]} \W[Q]=1;$$
i.e., the LHS defines a probability distribution. So, we abuse notation by also letting $\cW[u]$ denote the probability distribution over all walks in this set.
For a walk $Q=(v_i)_{i=0}^\ell$, we let 
$$ d(Q):=|\{v_i:0\leq i\leq \ell\}|$$
to denote the number of distinct vertices visited in $Q$.

\begin{lemma} \label{lem:expected-length-bound}
    %Let $P$ be a reversible random walk matrix on a (finite) graph $G = (V \cup B, E)$, and let $\nu$ be a stationary distribution such that $\nu(v) > 0$ for all $v \in V \cup B$. 
    For any $u \in V$ we have
    \[
        %\frac1{\nu(u)}
        \sum_{v \in V : v \neq u} \W_v[u \to v] = \E_{Q\sim \cW[u]}\left[d(Q) - 2\right] \leq \E_{Q\sim \cW[u]}[|Q|-1],
    \]
    where $|Q|$ is the length (number of edges) of $Q$.
\end{lemma}
\begin{proof}
%    \textcolor{blue}{We could probably do the whole proof in terms of $|Q|$ rather than summing over $k$ and $m$; do you think it's worth it?}\textcolor{red}{Yes, let's do that}
    % Throughout, given a walk $Q$ we will denote the vertices of the walk by $Q = (q_0,q_1,\ldots,q_{|Q|})$. Recall that $W_{\{u\}}[v \to u]$ is the set of walks which start at $v$ and end the first time $u$ is reached, without ever reaching any vertex in $B$. So, $W_{\{u\}}[v \to u] \subseteq W^+_{\{u\}}[v]$. 
    %
    First, we write, 
    \begin{align*}
         \E_{Q \sim \cW[u]}\left[d(Q)-2\right] &= \E_{Q\sim \cW[u]} \sum_{v\in Q: v\notin B \cup \{u\}} \W[Q_1Q_2],
    \end{align*}
    where the inner sum in the RHS is over the {\bf first occurrence} of any vertex $v\in Q$ (other than $u$ and the last vertex), $Q_1$ is the (prefix) of $Q$ which is a walk from $u\to v$ and $Q_2$ is the rest which is a walk from $v\to B$. Note that inner-sum in the RHS ranges over exactly $d(Q)-2$ vertices as we are excluding $u$ and the last vertex of $Q$.
    Changing the order of sums we can write,
    \begin{align*}
         \E_{Q \sim W[u]}[d(Q)-2] &= \sum_{v\neq u}\sum_{Q\in \cW_v[u\to v]}\W[Q]\cdot \sum_{Q\in \cW[v]} \W[Q]\\
        &=\sum_{v\neq u}\sum_{Q\in \cW_v[u\to v]}\W[Q]\cdot 1 = \sum_{v\neq u} \W_v[u\to v].
    \end{align*}
    where the second to last identity uses that the set of walks in $\cW[v]$ form a probability distribution.
    This proves the first identity in lemma's statement. The inequality simply follows from $d(Q)\leq |Q|+1$, i.e. the number of distinct vertices of $Q$ is at most one more than the number of edges of $Q$.
\end{proof}

\section{Local to Global Argument via $\cC$-Lorentzian Polynomials} \label{sec:C-Lor-argument}
In this section, we prove a sort of edge trickle-down theorem for partite simplicial complexes, based on the $\cC$-Lorentzian polynomial machinery. We first recall some notation from the previous sections which we will utilize here.

Throughout, we fix a partite simplicial complex $(X,\mu)$ with parts given by $S_v = \{vs_v \in X(1)\}$ for $v \in V$, and an absorbing (possibly directed) graph $G = (V \cup B, E)$ with boundary vertices $B$ which has no edges between boundary vertices. We further associate to $G$ a random walk matrix $W \in \R^{(V \cup B) \times (V \cup B)}$, where $W(u,v) > 0$ if and only if $u\sim v \in E$. Given any $\sigma \in X(d-2)$, we also denote by $P_\sigma$ the transition probability matrix on the skeleton of $X_\sigma$, weighted according to $\mu$.

Further, given $u,v \in V$ we let $G_{u,v}$ be the bipartite graph on $S_u \cup S_v$ with adjacency matrix $A_{u,v}$ with edge weights $A_{u,v}(us_u,vs_v)=\P_{\sigma\sim\mu}[us_u,vs_v\in \sigma]$, diagonal weighted degree matrix $D_{u,v}$, and transition probability matrix $P_{u,v}$ of the simple random walk on $G_{u,v}$.

The main goal of this section is then to prove the following result.

\begin{theorem}\label{thm:mainlorentzian}
    Suppose $(X,\mu)$ is connected, and suppose for all $\sigma\in X(d-2)$ we have 
    \begin{equation} \label{eq:quad-condition}
        \lambda_2(P_\sigma)\leq \sqrt{W(u,v) \cdot W(v,u)},
    \end{equation}
    where $u,v$ are the two distinct vertices in $V$ for which $u,v \not\in V(\sigma)$. Then we have
    \[
        \lambda_2(P_{u,v})\leq\sqrt{\W_u[v\to u] \cdot \W_v[u\to v]}
    \]
    for all distinct $u,v \in V$.
\end{theorem}
%\textcolor{red}{Do we need any more assumptions?} \textcolor{blue}{Added connectedness, I don't think we need anything else? But perhaps we can go through the proof in full detail over zoom.}

In the following sections we will prove the above theorem. To do this, we will at various points recall results and notation from \cite[Sec. 4]{LLO25}.

\subsection{Polynomials and $\pi$ Maps for Partite Complexes} \label{sec:polys-pi-maps}

In this section, we recursively define polynomials associated our given simplicial complex $X$, which we will later prove are $\cC$-Lorentzian with respect to particular cones. To construct these polynomials, we need to define a family commutative $\pi$ maps. We first recall the definition of commutativity.

\begin{definition}[Commutative $\pi$ maps] \label{def:pi-maps}
    Let $\bm\pi = \{\pi_{\sigma+\tau}: \R^{X_\sigma(1)} \to \R^{X_{\sigma \cup \tau}(1)}\}_{\sigma \in X,\tau \in X_\sigma}$ be a family of linear maps having the following property:
    \[
        \text{$\pi_{\sigma+\tau}(\bm{t})_H = t_H - L\big((t_x)_{x \in \tau}\big)$, for all  $H \in X_{\sigma \cup \tau}(1)$, where $L = L_{\sigma,\tau,H}$ is some linear functional.}
    \]
    We will also use the shorthands $\pi_{\sigma+x} := \pi_{\sigma+\{x\}}$ and $\pi_\sigma := \pi_{\varnothing+\sigma}$.
    We say that $\bm\pi$ is \textbf{commutative} if for all $\sigma \in X$, $\tau \in X_\sigma$, $\omega \in X_{\sigma \cup \tau}$, we have:
    \[
        \text{$\pi_{\sigma+(\tau \cup \omega)}(\bm{t}) = \pi_{(\sigma \cup \tau) + \omega}(\pi_{\sigma+\tau}(\bm{t}))$.}
    \]
\end{definition}

We now construct the commutative family of $\pi$ maps, which will lead to the polynomials we will use in our analysis. To do this, we will utilize the random walk notation of \cref{sec:random-walks-absorbing}. Given $\sigma \in X$ and $\{us_u,vs_v\} \in X_\sigma(2)$, we define
\[
    \pi_{\sigma + us_u}(\bm{t})_{vs_v} = t_{vs_v} - \W_{V(\sigma) \cup \{u\}}[v\to u] \cdot t_{us_u}.
\]
We now prove that these $\pi$ maps are commutative.

\begin{lemma} For all $\sigma \in X$ and $\{us_u,vs_v,ws_w\} \in X_\sigma(3)$, we have
\[
    \pi_{(\sigma \cup \{us_u\})+vs_v}(\pi_{\sigma + us_u}(\bm{t}))_{ws_w} = \pi_{(\sigma \cup \{vj\})+us_u}(\pi_{\sigma + vs_v}(\bm{t}))_{ws_w}.
\]
\end{lemma}
\begin{proof}
We compute
\begin{align*}
    &\pi_{(\sigma \cup \{us_u\})+vs_v}(\pi_{\sigma + us_u}(\bm{t}))_{ws_w} \\
        &= \pi_{\sigma+us_u}(\bm{t})_{ws_w} - \W_{V(\sigma) \cup \{u,v\}}[w \to v] \cdot \pi_{\sigma+us_u}(\bm{t})_{vs_v} \\
        &= (t_{ws_w} - \W_{V(\sigma) \cup \{u\}}[w \to u] \cdot t_{us_u}) - \W_{V(\sigma) \cup \{u,v\}}[w \to v] \cdot (t_{vs_v} - \W_{V(\sigma) \cup \{u\}}[v \to u] \cdot t_{us_u}) \\
        &= t_{ws_w} - \W_{V(\sigma) \cup \{u,v\}}[w \to v] \cdot t_{vs_v} - (\W_{V(\sigma) \cup \{u\}}[w \to u] - \W_{V(\sigma) \cup \{u,v\}}[w \to v] \cdot \W_{V(\sigma) \cup \{u\}}[v \to u]) \cdot t_{us_u}\\
        &= t_{ws_w} - \W_{V(\sigma) \cup \{u,v\}}[w \to v] \cdot t_{vs_v} - \W_{V(\sigma) \cup \{u,v\}}[w \to u] \cdot t_{us_u} \\
        &= \pi_{(\sigma \cup \{vj\})+ui}(\pi_{\sigma + vj}(\bm{t}))_{ws_w}.
\end{align*}
The second to last equality follows by \cref{fact:RWprop} for $U=V(\sigma) \cup \{u\}$.
\end{proof}

With this, we can now construct the polynomials associated to $X$. Note that the following definition applies to any simplicial complex, but here we will apply it specifically to partite complexes.

\begin{definition} \label{def:polys-X}
    We inductively define a family of polynomials $p_\sigma(\bm{t})$ on $\R^{X_\sigma(1)}$ of degree $d-|\sigma|$ associated to every $\sigma \in X$. If $\sigma$ is a facet of $X$ (so that $|\sigma| = d$) we define $p_\sigma(\bm{t}) = \mu(\sigma)$, and if $|\sigma| < d$ we define $p_\sigma(\bm{t})$ via
    \[
        (d-|\sigma|) \cdot p_\sigma(\bm{t}) = \sum_{x\in X_\sigma(1)} t_x \cdot p_{\sigma \cup \{x\}}(\pi_{\sigma+x}(\bm{t})).
    \]
    In particular, if $|\sigma| = d-1$ then $p_\sigma$ is the linear form given by
    \begin{equation}\label{eq:linearpcase}
        p_\sigma(\bm{t}) = \sum_{x \in X_\sigma(1)} t_x \cdot \mu(\sigma \cup \{x\}).
    \end{equation}
\end{definition}

In what follows we will prove that these polynomials $p_\sigma$ are Lorentzian with respect to certain associated cones. We next describe these cones.

\subsection{Cones of Positive Vectors} \label{sec:positive-cones}

In this section we define cones $\cC_\sigma$ such that $p_\sigma$ is $\cC_\sigma$-Lorentzian under condition \eqref{eq:quad-condition}. The cone $\cC_\sigma$ is called the cone of positive vectors, and we recall its definition now.

\begin{definition}[Cone of positive vectors] \label{def:pos-vector}
    Given $\sigma \in X$, we say that a vector $\bm{v} \in \R^{X_\sigma(1)}_{>0}$ is \textbf{$\bm\pi$-positive}  if $\pi_{\sigma+\tau}(\bm{v}) \in \R^{X_{\sigma \cup \tau}(1)}_{>0}$ for all $\tau \in X_\sigma$. We also say $\bm{v}$ is \textbf{$\bm\pi$-non-negative} if the inequalities are not strict.  We further define $\cC_\sigma$ to be the set of all $\bm\pi$-positive vectors. By definition of $\bm\pi$, it follows that $\cC_\sigma$ is an open convex (possibly empty) cone.
\end{definition}

The main goal of this section is then to prove some basic properties of the cones of non-negative vectors for the polynomials and $\pi$ maps defined above. To do this, we first define some special vectors in those cones called the $\alpha$ vectors. For $\sigma \in X$ and $w \in V \setminus V(\sigma)$, we define $\bm\alpha_\sigma(w) \in \R^{X_\sigma(1)}$ via
\[
    \alpha_\sigma(w)_{us_u} = \W_{V(\sigma)}[u \to w] \qquad \text{for all} ~ us_u \in X_\sigma(1).
\]
The following is the main result we will use regarding the $\alpha$ vectors.

\begin{lemma} \label{lem:alpha-pi}
    For any $\sigma \in X$, $\tau \in X_\sigma$, and $v \in V \setminus V(\sigma)$, we have
    \[
        \pi_{\sigma + \tau}(\bm\alpha_\sigma(v)) = \begin{cases}
            \bm\alpha_{\sigma \cup \tau}(v), & v \not\in V(\tau) \\
            \bm{0}, & v \in V(\tau).
        \end{cases}
    \]
    Furthermore, $\bm\alpha_\sigma(v)$ is $\bm\pi$-non-negative.
\end{lemma}
\begin{proof}
    By a standard induction argument, we only need to prove this for $\tau = \{us_u\}$. For all $zs_z \in X_{\sigma \cup \{us_u\}}(1)$, we compute
    \begin{align*}
        \pi_{\sigma+us_u}(\bm\alpha_\sigma(v))_{zs_z} &= \alpha_\sigma(v)_{zs_z} - \W_{V(\sigma)\cup\{u\}}[z \to u] \cdot \alpha_\sigma(v)_{us_u} \\
            &= \W_{V(\sigma)}[z \to v]
             - \W_{V(\sigma)\cup \{u\}}[z \to u] \cdot \W_{V(\sigma)}[u\to v] \\
            &\underset{\text{\cref{fact:RWprop}}}{=} \begin{cases}
                \W_{V(\sigma)\cup\{u\}}[z\to v], & u \neq v \\
                0, & u = v
            \end{cases} \\
            &= \begin{cases}
                \alpha_{\sigma \cup \tau}(v)_{zs_z}, & v \not\in V(\tau) \\
                0, & v \in V(\tau),
            \end{cases}
    \end{align*}
    as desired.
\end{proof}

This gives non-emptiness of the cones of non-negative vectors as an immediate corollary.

\begin{corollary} \label{cor:cones-nonempty}
    For all non-facets $\sigma \in X$, the cone $C_\sigma$ is non-empty.
\end{corollary}
\begin{proof}
    For all $\sigma \in X$ and $us_u \in X_\sigma(1)$, we have $\alpha_\sigma(u)_{u_s} = \W_{V(\sigma)}[u \to u] \geq 1$. Thus for all $\tau \in X_\sigma$ and $us_u \in X_{\sigma \cup \tau}(1)$ we have
    \[
        \pi_{\sigma+\tau}\left(\sum_{v \in V \setminus V(\sigma)} \bm\alpha_\sigma(v)\right)_{us_u} \underset{\text{\cref{lem:alpha-pi}}}{=} \sum_{v \in V \setminus V(\sigma \cup \tau)} \alpha_{\sigma \cup \tau}(v)_{us_u} \geq \alpha_{\sigma \cup \tau}(u)_{us_u} \geq 1 > 0,
    \]
    and therefore $\sum_{v \in V \setminus V(\sigma)} \bm\alpha_\sigma(v) \in C_\sigma$.
\end{proof}

\subsection{Polynomials $p_\sigma$ are $\cC_\sigma$-Lorentzian}

We are now ready to prove that the polynomials $p_\sigma$ are $\cC_\sigma$-Lorentzian under condition \eqref{eq:quad-condition}. To do this, we need the following theorem, along with a lemma which puts into context the main required condition of the theorem.

\begin{theorem}[Connected + Quadratics $\implies$ Lorentzian] \label{thm:conn+quad=Lor}
    Suppose $(X,\mu)$ is a $d$-dimensional connected simplicial complex and $\cC_\varnothing$ is non-empty. If the Hessian of $p_\sigma$ has at most one positive eigenvalue for all $\sigma \in X(d-2)$, then $p_\sigma$ is $\cC_\sigma$-Lorentzian for all $\sigma \in X$.
\end{theorem}

\begin{lemma} \label{lem:quadratic-entries}
    Fix $\sigma \in X(d-2)$, and let $A$ be the adjacency matrix of the $1$-skeleton of $X_\sigma$, weighted according to $\mu$. That is, for all $\{x,y\} \in X_\sigma(2)$ the entry $A(x,y)$ is $\mu(\sigma \cup \{x,y\})$. Further, let $D$ be the diagonal matrix with entries given by
    \[
        D(x,x) = \sum_{y \in X_{\sigma \cup \{x\}}(1)} A(x,y) \cdot \partial_{t_x} \pi_{\sigma+x}(\bm{t})_y = \nabla\left[A \cdot \pi_{\sigma+x}(\bm{t})\right]_x.
    \]
    Then the Hessian of $p_\sigma$ is $A+D$.
\end{lemma}

We now apply \cref{thm:conn+quad=Lor} to prove that $p_\sigma$ is $C_\sigma$-Lorentzian for all $\sigma \in X$. To do this, we just need to prove that the Hessian of $p_\sigma$ has at most one positive eigenvalue for all $\sigma \in X(d-2)$.

\begin{theorem} \label{thm:partite-Lorentzian}
    Let $(X, \mu)$ be connected, and suppose that for every $\sigma\in X(d-2)$ we have 
    \[
        \lambda_2(P_\sigma) \leq \sqrt{W(u,v) \cdot W(v,u)},
    \]
    where $u,v$ are the two distinct vertices such that $u,v \not\in V(\sigma)$. Then $p_\sigma$ is $C_\sigma$-Lorentzian for all $\sigma \in X$.
\end{theorem}
\begin{proof}
    By \cref{thm:conn+quad=Lor} and the fact that $C_\sigma$ is non-empty by \cref{cor:cones-nonempty}, we just need to prove that the Hessian of $p_\sigma$ has at most one positive eigenvalue for all $\sigma \in X(d-2)$. Given $\sigma \in X(d-2)$, to compute the Hessian of $p_\sigma$, we use \cref{lem:quadratic-entries}. Let $u,v$ be the two vertices such that $u,v \not\in V(\sigma)$. For $\{us_u,vs_v\} \in X_\sigma(2)$ compute
    \[
        \partial_{t_{us_u}} \pi_{\sigma + us_u}(\bm{t})_{vs_v} = -\W_{V(\sigma) \cup \{u\}}[v \to u] = -W(v,u),
    \]
    where the last equality follows from the fact that every vertex adjacent to $v$ is contained in $V(\sigma) \cup \{u\} \cup B$. Now let $A_\sigma$ be the adjacency matrix of the skeleton of $X_\sigma$, weighted according to $\mu$, and let $M_\sigma$ be the diagonal matrix with entries given by
    \[
        M_\sigma(us_u,us_u) = W(v,u) \sum_{vs_v \in X_{\sigma \cup \{us_u\}}(1)} A_\sigma(us_u,vs_v) \qquad \text{for all} ~ us_u \in X_\sigma(1)
    \]
    and
    \[
        M_\sigma(vs_v,vs_v) = W(u,v) \sum_{us_u \in X_{\sigma \cup \{vs_v\}}(1)} A_\sigma(vs_v,us_u) \qquad \text{for all} ~ vs_v \in X_\sigma(1),
    \]
    By \cref{lem:quadratic-entries}, the Hessian of $p_\sigma$ is equal to $A_\sigma - M_\sigma$. Letting $D_\sigma$ be the diagonal degree matrix weighted according to $\mu$, we have
    \[
        M_\sigma = \begin{bmatrix}
            W(v,u) \cdot I_{us_u \in X_\sigma(1)} & 0 \\
            0 & W(u,v) \cdot I_{vs_u \in X_\sigma(1)}
        \end{bmatrix} \cdot D_\sigma,
    \]
    where $I_S$ is the identity matrix indexed by the elements of $S$. Now by assumption we have
    \[
        \lambda_2(D_\sigma^{-1}A_\sigma) = \lambda_2(P_\sigma) \leq \sqrt{W(v,u) \cdot W(u,v)},
    \]
    and thus by \cref{lem:eigenvaluebipartite} we have that
    \[
        D_\sigma^{-1/2} A_\sigma D_\sigma^{-1/2} - \begin{bmatrix}
            W(v,u) \cdot I_{\{us_u \in X_\sigma(1)\}} & 0 \\
            0 & W(u,v) \cdot I_{\{vs_v \in X_\sigma(1)\}}
        \end{bmatrix}
        = D_\sigma^{-1/2} \left(A_\sigma - M_\sigma\right) D_\sigma^{-1/2}
    \]
    has at most one positive eigenvalue. Therefore the Hessian $A_\sigma - M_\sigma$ has at most one positive eigenvalue.
\end{proof}

\subsection{Special Directional Derivatives of $p_\sigma$} \label{sec:dir-derivs}

In this section, we will prove a nice formula for directional derivatives in $\bm\alpha(v)$ directions (for various $v \in V$) applied to $p_\sigma$. This will produce a quadratic polynomial with Hessian related to the random walk matrix $P_{u,v}$. Since the $\bm\alpha(v)$ vectors are $\bm\pi$-non-negative, this Hessian will have one positive eigenvalue, allowing us to analyze the eigenvalues of $P_{u,v}$. First we recall a few lemmas which will make out computations much simpler.

\begin{lemma}[cf. Lem.~3.3 of \cite{BL21}] \label{lem:derivs-expression}
    Given $\sigma \in X$ and $x \in X_\sigma(1)$, we have
    \[
        \partial_{t_x} p_\sigma(\bm{t}) = p_{\sigma \cup \{x\}}(\pi_{\sigma+x}(\bm{t})).
    \]
    % which by Euler's formula implies
    % \[
    %     (d-|\sigma|) \cdot p_\sigma(\bm{t}) = \sum_{x \in X_\sigma(1)} t_x \cdot p_{\sigma \cup \{x\}}(\pi_{\sigma+x}(\bm{t})).
    % \]
    Further, for $x,y \in X_\sigma(1)$ with $x \neq y$, we have
    \[
        \partial_{t_y} \partial_{t_x} p_\sigma(\bm{t}) = \begin{cases}
            p_{\sigma \cup \{x,y\}}(\pi_{\sigma+\{x,y\}}(\bm{t})), & \{x,y\} \in X_\sigma(2) \\
            0, & \text{otherwise}.
        \end{cases}
    \]
\end{lemma}

\begin{lemma}[Chain rule] \label{lem:chain-rule}
    If $f \in \R[t_1,\ldots,t_n]$ is a polynomial, $L: \R^m \to \R^n$ is a linear map, and $\bm{v} \in \R^m$ then
    \[
        \nabla_{\bm{v}}[f(L(\bm{t}))] = [\nabla_{L(\bm{v})}f](L(\bm{t})).
    \]
\end{lemma}

Using these two lemmas, we now prove the nice formula for applying $\bm\alpha(v)$ directional derivatives to $p_\sigma$.

\begin{lemma} \label{lem:dir-deriv-expression}
    Given $\sigma \in X$, let $\{v_1,\ldots,v_m\}$ be any ordering of the vertices of $V \setminus V(\sigma)$. %let $B_\sigma = \{b_v \in B : v \not\in V(\sigma)\} = \{b_{v_1},\ldots,b_{v_m}\}$. 
    We have
    \[
        \left[\prod_{v\notin V(\sigma)} \nabla_{\bm\alpha(v)}\right] p_\sigma(\pi_\sigma(\bm{t})) = \left[\prod_{i=1}^m \W_{V(\sigma) \cup \{v_1,\ldots,v_{i-1}\}}[v_i \to v_i]\right] \sum_{\tau \in X(d): \sigma \subseteq \tau} \mu(\tau).
    \]
    In particular, $\prod_{i=1}^m \W_{V(\sigma) \cup \{v_1,\ldots,v_{i-1}\}}[v_i \to v_i]$ is independent of the order of $\{v_1,\ldots,v_m\}$.
\end{lemma}
\begin{proof}
    We prove this by induction on $\codim(\sigma)$ %$B_\sigma$, 
    where the result is trivial in the case that $\codim(\sigma) = 0$. For $v\notin V(\sigma)$, we first compute an expression for a single directional derivative.
    \begin{align*}
        \nabla_{\bm\alpha(v)} [p_\sigma(\pi_\sigma(\bm{t}))] &\underset{\text{chain rule}}{=} [\nabla_{\pi_\sigma(\bm\alpha(v))} p_\sigma](\pi_\sigma(\bm{t})) \\
            &\underset{\text{\cref{lem:alpha-pi}}}{=} [\nabla_{\bm\alpha_\sigma(v)} p_\sigma](\pi_\sigma(\bm{t})) \\
            &\underset{\text{\cref{lem:derivs-expression}}}{=} \sum_{ws_w\in X_\sigma(1)} \alpha_\sigma(v)_{ws_w} \cdot p_{\sigma \cup \{ws_w\}}(\pi_{\sigma + ws_w}(\cdot))](\pi_\sigma(\bm{t})) \\
            &\underset{\text{commut. of $\bm\pi$}}{=} \sum_{ws_w \in X_\sigma(1)} \alpha_\sigma(v)_{ws_w} \cdot p_{\sigma \cup \{ws_w\}}(\pi_{\sigma \cup \{ws_w\}}(\bm{t})).
    \end{align*}
    To prove the desired expression, we now want to apply $\nabla_{\bm\alpha(u)}$ for all other $u \not\in V(\sigma)$ besides $v$. A priori, it seems this computation could become quite complicated. However, we now make a small observation which will simplify things significantly. Note that by \cref{lem:alpha-pi} we have
    \[
        \pi_{\sigma \cup \{us_u\}}(\bm\alpha(u)) = \bm{0},
    \]
    which implies
    \begin{equation} \label{eq:alpha-order}
        \nabla_{\bm\alpha(u)} [p_{\sigma \cup \{us_u\}}(\pi_{\sigma \cup \{us_u\}}(\bm{t}))] \underset{\text{chain rule}}{=} [\nabla_{\pi_{\sigma \cup \{ui\}}(\bm\alpha(u))} p_{\sigma \cup \{us_u\}}](\pi_{\sigma \cup \{us_u\}}(\bm{t})) = 0.
    \end{equation}
    Therefore we have
    \begin{align*}
        \left[\prod_{u \not\in V(\sigma)} \nabla_{\bm\alpha(u)}\right] p_\sigma(\pi_\sigma(\bm{t})) &= \left[\prod_{u \not\in V(\sigma) \cup \{v\}} \nabla_{\bm\alpha(u)}\right] \nabla_{\bm\alpha(v)} [p_\sigma(\pi_\sigma(\bm{t}))] \\
            &= \sum_{ws_w \in X_\sigma(1)} \alpha_\sigma(v)_{ws_w} \left[\prod_{u \not\in V(\sigma) \cup \{v\}} \nabla_{\bm\alpha(u)}\right] p_{\sigma \cup \{ws_w\}}(\pi_{\sigma \cup \{ws_w\}}(\bm{t})) \\
            &\underset{\text{\eqref{eq:alpha-order}}}{=} \sum_{s_v: vs_v \in X_\sigma(1)} \alpha_\sigma(v)_{vs_v} \left[\prod_{u \not\in V(\sigma) \cup \{v\}} \nabla_{\bm\alpha(u)}\right] p_{\sigma \cup \{vs_v\}}(\pi_{\sigma \cup \{vs_v\}}(\bm{t})).
    \end{align*}
    That is, the sum of the RHS is only over $vs_v \in X_\sigma(1)$ for the specific vertex $v$ with which we applied the directional derivative $\nabla_{\bm\alpha(v)}$. Now suppose we order vertices not in $V(\sigma)$ as $v_1,\dots, v_m$ for $m=\codim(\sigma)$. Recalling that $\alpha_\sigma(v)_{vs_v} = \W_\sigma[v \to v]$, by induction we finally have
    \begin{align*}
        \left[\prod_{i=1}^m \nabla_{\bm\alpha(v_i)}\right] &p_\sigma(\pi_\sigma(\bm{t})) \\
            &= \sum_{s:v_1s \in X_\sigma(1)} \alpha_\sigma(v_1)_{v_1s} \left[\prod_{i=2}^m \nabla_{\bm\alpha(v_i)}\right] p_{\sigma \cup \{v_1s\}}(\pi_{\sigma \cup \{v_1s\}}(\bm{t})) \\
            &\underset{\text{IH}}{=} \sum_{s:v_1s \in X_\sigma(1)} \W_\sigma[v_1 \to v_1] \cdot \left[\prod_{i=2}^m \W_{V(\sigma) \cup \{v_1,\ldots,v_{i-1}\}}[v_i \to v_i]\right] \sum_{\tau \in X(d): \sigma \cup \{v_1s\} \subseteq \tau} \mu(\tau) \\
            &= \left[\prod_{i=1}^m \W_{V(\sigma) \cup \{v_1,\ldots,v_{i-1}\}}[v_i \to v_i]\right] \sum_{\tau \in X(d): \sigma \subseteq \tau} \mu(\tau)
    \end{align*}
    as desired.
\end{proof}

\subsection{Proof of \cref{thm:mainlorentzian}}

We now finally use \cref{lem:dir-deriv-expression} to prove \cref{thm:mainlorentzian}. First we compute a certain Hessian related to the random walk matrix $P_{u,v}$. Let $\sigma = \{us_u,vs_v\}$ and let $w_1,\dots,w_{d-2}$ be the rest of the vertices in $V$. %$B_\sigma = \{b_{w_1},\ldots,b_{w_{d-2}}\}$. 
We have
\[
\begin{split}
    \nabla_{\bm\alpha(w_1)} \cdots &\nabla_{\bm\alpha(w_{d-2})} \partial_{us_u} \partial_{vs_v} p_\varnothing(\bm{t}) \\
        &\underset{\text{\cref{lem:derivs-expression}}}{=} \nabla_{\bm\alpha(w_1)} \cdots \nabla_{\bm\alpha(w_{d-2})} p_{\{us_u,vs_v\}}(\pi_{\{us_u,vs_v\}}(\bm{t})) \\
        &\underset{\text{\cref{lem:dir-deriv-expression}}}{=} \left[\prod_{i=1}^{d-2} \W_{\{u,v,w_1,\dots,w_{i-1}\}}[w_i \to w_i]\right] \sum_{\tau \in X(d): \{us_u,vs_v\} \subseteq \tau} \mu(\tau)
\end{split}
\]
and
\[
\begin{split}
    \nabla_{\bm\alpha(w_1)} \cdots &\nabla_{\bm\alpha(w_{d-2})} \partial_{us_u}^2 p_\varnothing(\bm{t}) \\
        &\underset{\text{\cref{lem:derivs-expression}}}{=} \sum_{s: vs \in X_{\{us_u\}}(1)} \partial_{t_{us_u}} \pi_{\{us_u\}}(\bm{t})_{vs} \cdot \nabla_{\bm\alpha(w_1)} \cdots \nabla_{\bm\alpha(w_{d-2})} [p_{\{us_u,vs\}}(\pi_{\{us_u,vs\}}(\bm{t}))] \\
        &\underset{\text{\cref{lem:dir-deriv-expression}}}{=} \sum_{s: vs \in X_{\{us_u\}}(1)} -\W_{u}[v \to u] \cdot \left[\prod_{i=1}^{d-2} \W_{\{u,v,w_1,\ldots,w_{i-1}\}}[w_i \to w_i]\right] \sum_{\tau \in X(d): \{us_u,vs\} \subseteq \tau} \mu(\tau) \\
        &= -\W_{u}[v \to u] \cdot \left[\prod_{i=1}^{d-2} \W_{\{u,v,w_1,\ldots,w_{i-1}\}}[w_i \to w_i]\right] \sum_{\tau \in X(d) : us_u \in \tau} \mu(\tau).
\end{split}
\]
Thus up to positive scalar, the Hessian $H_{u,v}$ of $\nabla_{\bm\alpha(w_1)} \cdots \nabla_{\bm\alpha(w_{d-2})} p_\varnothing(\bm{t})$ is equal to $A_{u,v} - D_{u,v} M_{u,v}$ where $A_{u,v}$ and $D_{u,v}$ are as defined above \cref{thm:mainlorentzian} and $M_{u,v}$ is the diagonal matrix which equals $\W_{u}[v \to u]$ for the $u$ entries and equals $\W_{v}[u \to v]$ for the $v$ entries. Since $p_\varnothing$ is $\cC_\varnothing$-Lorentzian by \cref{thm:partite-Lorentzian}, and $\bm\alpha(w_k) \in \overline{C}_\varnothing$ for all $k$ by \cref{lem:alpha-pi}, the matrix $H_{u,v}$ has one positive eigenvalue by \cref{lem:main-purpose-lemma}. Since $G_{u,v}$ is a bipartite graph, by \cref{lem:eigenvaluebipartite} we finally have that $\lambda_2(P_{u,v}) = \lambda_2(D_{u,v}^{-1} A_{u,v}) \leq \sqrt{\W_{u}[v \to u] \cdot \W_{v}[u \to v]}$. This completes the proof of \cref{thm:mainlorentzian}.

\section{From $\cC$-Lorentzian to Spectral Independence} \label{sec:spectral-indep}
In this section we prove \cref{thm:maintechnical} and then we use it to prove \cref{cor:maintechnical-distinct}. We then use \cref{cor:maintechnical-distinct} to prove \cref{thm:main}.
\begin{proof}[Proof of \cref{thm:maintechnical}]
   By the conclusion of \cref{thm:mainlorentzian}, for any pair of vertices $u,v\in V$ we have 
\begin{equation}\label{eq:Puvlambda2}\lambda_2(P_{u,v}) \leq \sqrt{\W_{u}[v \to u] \cdot \W_{v}[u \to v]},
\end{equation}
where recall that $P_{u,v}$ is the simple random walk matrix on the weighted bipartite graph $G_{u,v}$ on $S_u\cup S_v$ where the weight of every edge $A_{u,v}(us_u,vs_v)$ is the probability that $us_u,vs_v$ are in a $\sigma\sim\mu$. Observe that from \eqref{eq:Puvlambda2} we have $\lambda_2(P_{u,v})\leq 0$ if $u,v$ there is no directed path from $u$ to $v$ (or from $v$ to $u$) w.r.t. the random walk matrix $W$. Thus recalling that $M_\varnothing(u,v) = \sqrt{\W_{u}[v \to u] \cdot \W_{v}[u \to v]}$, the assumptions of \cref{lem:dpartiteeigenvalues} are satisfied by $M_\varnothing$. Therefore we have
\[
    \lambda(P_\varnothing) \leq \frac{\lambda_{\max}(M_\varnothing)}{d-1}
\]
as desired.

To prove \cref{thm:maintechnical} for any $\tau \in X$ of codimension at least 2, notice that $(X_\tau, \mu_\tau)$ is a connected $\codim(\tau)$-partite complex with parts indexed by $V \setminus V(\tau)$. Thus by adding $V(\tau)$ to the set of boundary vertices $B$, we consider a random walk on $V_\tau \cup B_\tau$ for $V_\tau = V \setminus V(\tau)$ and $B_\tau = B \cup V(\tau)$, with the same transition probabilities as $W$. Thus we can define $M_\tau$ analogously to $M_\varnothing$ by setting $M_\tau(u,v) = \sqrt{\W_{u}[v \to u] \cdot \W_{v}[u \to v]}$ according to this new random walk matrix on $V_\tau \cup B_\tau$. Further, the assumptions of \cref{lem:dpartiteeigenvalues} are satisfied by $M_\tau$. Therefore we have
\[
    \lambda(P_\tau) \leq \frac{\lambda_{\max}(M_\tau)}{\codim(\tau)-1}
\]
for all $\tau \in X$ of codimension at least 2, as desired.
\end{proof}

\subsection{Upper Bounds via the Number of Distinct Walk Vertices}
\begin{proof}[Proof of \cref{cor:maintechnical-distinct}]
    To prove this corollary, we just need to show that
    \[
        \lambda_{\max}(M_\tau) \leq \max_{u \in V \setminus V(\tau)} \E_{Q\sim \cW_{V(\tau)}[u]}[d(Q)-2]
    \]
    Given any $\tau \in X$ of codimension at least 2, let $M'_\tau$ be defined as $M'_\tau(u,v)=\W_{V(\tau) \cup \{v\}}[u\to v]$ for all $u,v\in V \setminus V(\tau)$, so that $\bar{M}_\tau' = M_\tau$ (using the $\bar{M}$ notation from \cref{sec:lin-alg}). By \cref{lem:rhoaverage}, we have
    $$ \lambda_{\max}(M_\tau) = \rho(M_\tau)\leq\rho(M'_\tau),$$
    and then we can bound the spectral radius by the maximum row sum via
    \[
        \lambda_{\max}(M_\tau)\leq\rho(M'_\tau)\leq \max_{u \in V \setminus V(\tau)} \sum_{v \in V \setminus V(\tau) : v \neq u} \W_v[u\to v].
    \]
    Therefore, by \cref{lem:expected-length-bound} we have
    \[
        \lambda_{\max}(M_\tau) \leq \max_{u \in V \setminus V(\tau)} \sum_{v \in V \setminus V(\tau) : v \neq u} \W_v[u\to v] = \max_{u \in V \setminus V(\tau)} \E_{Q\sim \cW_{V(\tau)}[u]}[d(Q)-2]
    \]
    as desired.
    % We now upper-bound the RHS. For an arbitrary vertex $u\in V \setminus V(\tau)$, we have
    % \begin{align*}
    %     \E_{Q\sim \cW_{V(\tau)}[u]}[|Q|-1] = \sum_{k=2}^\infty \P[|Q|\geq k]\leq \sum_{k=2}^\infty (1-\eps)^k = \frac{1-\eps}{\eps}
    % \end{align*}
\end{proof}

\subsection{Generalizing Dobrushin's Condition}
\begin{proof}[Proof of \cref{thm:main}]
Given a $d$-partite complex $X$ corresponding to a multi-state spin system and the pairwise spectral influence matrix $\cI$, we first construct an undirected graph $G$ with corresponding random walk matrix $W$ and then we apply \cref{cor:maintechnical-distinct}. We add a boundary vertex $b_v$ for any $v\in V$, we let $B = \{b_v : v \in V\}$ and we let $V(G)=V\cup B$. For any $u,v\in V$ with $\cI(u,v)>0$ we include an edge $\{u,v\}$; we also connect any $b_v$ to $v$. Next, we determine the weight of the edges.

 Let $V_1 \cup B_1, V_2 \cup B_2, \ldots, V_\ell \cup B_\ell$ be the connected components of $G$. Furthermore, let $\epsilon_1,\epsilon_2,\dots,\epsilon_\ell$ be such that $\lambda_{\max}(\cI_i) = 1-\epsilon_i$ for all $i \in [\ell]$, where $\cI_i$ is the principal submatrix of $\cI$ corresponding to $G[V_i\cup B_i]$. Let $\bm{x}_i\in \R^{V_i}_{>0}$ be the corresponding eigenvectors which can be chosen to be strictly positive by the Perron-Frobenius theorem; i.e., $\cI_i \bm{x}_i=(1-\eps_i)\bm{x}_i$. %We further assume $\sum_{u \in V} x_i(u) = 1$ for all $i$.

We now define $W\in \R^{(V\cup B)\times (V\cup B)}$. %by defining  a block diagonal matrix with blocks $P_i$ corresponding to $V_i\cup B_i$ via
For any $u,v\in V_i$, we let
\[
    W(u,v) = \cI(u,v)\cdot \frac{x_i(v)}{x_i(u)}, \quad \text{for all } u,v \in V_i.
\]
Further, for $v\in V_i$ we let $W(v,b_v)=\eps_i, W(b_v,v)=1$, and we set every other entry of $W$ to be zero. 

Next, we show that indeed $W$ is a random walk matrix. First notice $W(u,v) \geq 0$ for all $u,v \in V\cup B$. For any $u \in V_i$ we have
\[
    \sum_{v \in V\cup B} W(u,v) = W(u,b_u)+\sum_{v \in V_i} W(u,v) = \eps_i+\sum_{v \in V_i} \cI(u,v) \cdot \frac{x_i(v)}{x_i(u)} = \eps_i+\frac{(1-\epsilon_i) \cdot x_i(u)}{x_i(u)} = 1.
\]
So, $W$ is stochastic.

Note that the graph $G$ that we constructed above is absorbing as there are at least one boundary vertex $b$ in every connected component. 
Now we check the assumption of \cref{cor:maintechnical-distinct}.
Fix $u,v\in V_i$ let $\tau\in X$ be a face of codimension 2 such that $u,v\notin V(\tau).$ Then by the definition of $\cI$, we have
$$\lambda_2(P_\tau)\leq \cI(u,v) = \sqrt{\cI(u,v)\cI(v,u)}=\sqrt{W(u,v)W(v,u)}.
$$
%where we use that $\cI(u,v)=\cI(v,u)$.
Since $X$ is connected, by \cref{cor:maintechnical-distinct}  we have
\[
    \lambda_2(P_\sigma) \leq \frac{1}{\codim(\sigma)-1} \max_u \E_{Q\sim \cW_{V(\sigma)}[u]}[d(Q)-2]\leq\frac{1}{\codim(\sigma)-1}\max_u \E_{Q\sim \cW_{V(\sigma)}[u]}[|Q|-1]
\]
for every $\sigma \in X$ of codimension at least 2.
% Since $X$ is connected, by \cref{thm:maintechnical} we have that 
% $$ \lambda_2(P_\varnothing)\leq  \frac{1}{d-1} \lambda_{\max}(M_\varnothing).$$
% Let $M'_\varnothing$ be defined as $M'_\varnothing(u,v)=\W_v[u\to v]$, for all $u,v\in V$. Then by \cref{lem:rhoaverage}, 
% $$ \rho(M_\varnothing)\leq\rho(M'_\varnothing)\leq \max_u \sum_{v:\neq u} \W_v[u\to v].$$
%and thus by \cref{lem:expected-length-bound} we have
%$$\lambda_2(P_\varnothing).$$
We now upper-bound the RHS. For an arbitrary vertex $u\in V \setminus V(\sigma)$, we have
\begin{align*}
    \E_{Q\sim \cW_{V(\sigma)}[u]}  [|Q|-1] = \sum_{k=2}^\infty \P[|Q|\geq k]\leq \sum_{k=1}^\infty (1-\eps)^k = \frac{1-\eps}{\eps}
\end{align*}
where in the first identity we used that every walk $Q$ has length at least 1. The inequality uses that from every vertex $v\in V_i$ we jump to an absorbing vertex $b_v$ with probability at least $\eps_i\geq \eps$. 
This then implies $\lambda_2(P_\sigma)\leq \frac{1-\eps}{(\codim(\sigma)-1)\eps}.$ Therefore $\mu$ is $\frac{1-\eps}{\epsilon}$-spectrally independent.
%
% Now, consider any face $\tau\in X$ of codimension at least 2. Let $\cI_\tau$ be the corresponding pairwise spectral influence matrix. The main observation is that by the definition of the $\cI$, for any $u,v\in X_\tau(1)$ we have $\cI_\tau(u,v)\leq \cI(u,v).$ Let $U\subseteq V$ be the vertices in the link of $\tau$ and let $\cI_U$ be the principal sub-matrix of $\cI$ on vertices in $U$. It follows by the Cauchy interlacing theorem that
% $$ \lambda_{\max}(\cI_U)\leq \lambda_{\max}(\cI).$$
% Finally, by \cref{lem:monotoneeigenvalues} we have that $\lambda_{\max}(\cI_\tau)\leq \lambda_{\max}(\cI_U)\leq 1-\eps$. Therefore, by a similar argument as above we obtain that $\lambda_2(P_\tau)\leq \frac{1-\eps}{(\codim(\tau)-1)\eps}.$
\end{proof}

\subsection{Path Complexes} \label{sec:path-complexes}

In this section, we use the machinery of this paper to prove a generalization of one of the main results of \cite{LLO25} for path complexes. First let us recall some notation from \cite{LLO25}. A $d$-partite simplicial complex $X$ is called a {\bf top-link path complex}\footnote{The definition we give here is slightly different than that of \cite{LLO25}, but the two definitions are equivalent.} if the parts/sites $V$ of $X$ can be labeled by $[d] = \{1,2,\ldots,d\}$ so that the following holds: for every $\tau \in X$ of codimension 2 such that $V \setminus V(\tau) = \{i,j\}$ for non-consecutive $i,j$, we have $\lambda_2(P_\tau) \leq 0$. The main trickle-down theorem obtained for top-link path complexes is then given as follows.

\begin{theorem}[Theorem 1.4 of \cite{LLO25}] \label{thm:path-complex}
    If $(X, \mu)$ is a connected $d$-partite top-link path complex such that $\lambda_2(P_\tau) \leq \frac{1}{2}$ for all $\tau \in X$ of codimension 2, then $\lambda_2(P_\tau) \leq \frac{1}{2}$ for all $\tau \in X$.
\end{theorem}

To reprove this theorem using our machinery, we first translate the definition of top-link path complex into the language of this paper. Given any top-link path complex, we can set $\cI(i,j) = 0$ for all $i,j$ which are non-consecutive integers. Thus our pairwise spectral influence matrix $\cI$ is a symmetric tri-diagonal matrix with zero diagonal (assuming the sites are ordered via the labels $[d]$). The above theorem then says that if all non-zero entries of $\cI$ are equal to $\frac{1}{2}$, then $\lambda_2(P_\tau) \leq \frac{1}{2}$ for all $\tau \in X$.

Note that if we could assume for some constant $\epsilon > 0$ that $\lambda_2(P_\tau) \leq \frac{1-\epsilon}{2}$ for all $\tau$ of codimension 2, then we could simply bound $\lambda_{\max}(\cI)$ by the maximum row sum of $\cI$ and apply \cref{thm:main} to obtain $\lambda_2(P_\tau) \leq \frac{1-\eps}{(\codim(\tau)-1) \epsilon}$ for all $\tau \in X$. (In fact, this recovers a weaker form of Theorem 1.21 of \cite{LLO25}.) Thus to reprove \cref{thm:path-complex}, we will instead use \cref{thm:maintechnical}.

Consider a path graph with $V = [d]$ and two boundary vertices $B = \{0,d+1\}$ adjacent to $1$ and $d$ respectively, and let the (non-directed) edges be weighted by $\frac{1}{2}$. Thus for all adjacent $i,j \in [d]$ we have $W(i,j) = \frac{1}{2}$.
% Further, for any $i,j \in [d]$ such that $i < j$, it is straightforward to compute
% \[
%     \W_j[i \to j] = \frac{i}{j} \quad \text{and} \quad \W_i[j \to i] = \frac{d+1-j}{d+1-i}.
% \]
% With this, let $M_\varnothing$ be the real symmetric matrix with zero diagonal and off-diagonal entries given by
% \[
%     M_\varnothing(i,j) = \sqrt{\W_j[i \to j] \cdot \W_i[j \to i]} = \min\left\{\sqrt{\frac{i(d+1-j)}{j(d+1-i)}}, \sqrt{\frac{j(d+1-i)}{i(d+1-j)}}\right\}.
% \]
We let $M'_\varnothing$ be a (non-symmetric) real matrix with zero diagonal and off-diagonal entries given by $M'_\varnothing(i,j) = \W_i[j \to i]$.
% \[
%     M'_\varnothing(i,j) = \W_i[j \to i] = \min\left\{\frac{j}{i}, \frac{d+1-j}{d+1-i}\right\}.
% \]
Since $\bar{M}'_\varnothing = M_\varnothing$ (using the $\bar{M}$ notation from \cref{sec:lin-alg}), we have that \cref{lem:rhoaverage} implies $\lambda_{\max}(M_\varnothing) \leq \lambda_{\max}(M'_\varnothing)$ by Perron-Frobenius. We then bound $\lambda_{\max}(M'_\varnothing)$ by bounding the maximum row sum as follows. For all $i \in [d]$ we have
\[
    \sum_{j} M'_\varnothing(i,j) = \sum_{j \neq i} \W_i[j \to i] = \sum_{j=1}^{i-1} \W_i[j\to i] + \sum_{j=i+1}^{d} \W_i[j\to i] = \frac{d-1}{2}.
    %\sum_{j < i} \frac{j}{i} + \sum_{j > i} \frac{d+1-j}{d+1-i} = \frac{i-1}{2} + \frac{d-i}{2} = \frac{d-1}{2},
\]
To see the last equality, observe  that by symmetry $\sum_{j=1}^{i-1}\W_i[j\to 0] = \sum_{j=1}^{i-1}\W_i[j\to i]$ and they sum up to $i-1$. So, each must be exactly $\frac{i-1}{2}$. Similarly, we have $\sum_{j=i+1}^d \W_i[j\to i]=\frac{d-i}{2}$. 
This implies $\lambda_{\max}(M'_\varnothing) \leq \frac{d-1}{2}$.

Therefore applying \cref{thm:maintechnical} gives: if for all $\sigma \in X(d-2)$ we have
\[
    \lambda_2(P_\sigma) \leq \sqrt{W(i,j) \cdot W(j,i)} = \begin{cases}
        \frac{1}{2}, & \text{$i,j$ are adjacent} \\
        0, & \text{$i,j$ are not adjacent},
    \end{cases}
\]
where $i,j$ are the two distinct sites in $V = [d]$ which have no spin in $\sigma$, then
\[
    \lambda_2(P_\varnothing) \leq \frac{\lambda_{\max}(M_\varnothing)}{d-1} \leq \frac{\lambda_{\max}(M'_\varnothing)}{d-1} \leq \frac{1}{2}.
\]
By the definition of top-link path complex, this is precisely \cref{thm:path-complex} for $\tau = \varnothing$. The general result for all $\tau \in X$ then follows from the fact that $(X_\tau, \mu|_{X_\tau})$ is also a connected top-link path complex (since $(X,\mu)$ is) with the same codimension 2 eigenvalue bounds.

\section{Future Directions}
Given our improved Dorbrushin's condition with a spectral pairwise influence matrix, the first question is to find new applications of this statement which were not approachable using the previous machineries. We expect to see many such directions in the next few years.

Another natural question stemming from our machinery is to prove a more general trickle-down theorem for non-partite simplicial complexes. We intend to do this in future work; that is, to reprove \cref{thm:oppenheim} and then to extend beyond worst case bounds on second eigenvalues of top links. In fact, a slight variation of the machinery of this paper can already be used to reprove \cref{thm:oppenheim}, but it is currently unclear how to generalize this idea.

Finally, we conclude this paper with a natural conjecture stemming from our family of Lorentzian polynomials.
Given a  graph $G$ with $d=|V|$ vertices, maximum degree $\Delta>1$, and $q$ colors, it is conjectured that the natural Glauber dynamics mixes in $\text{poly}(d)$ steps and generates a uniformly random proper coloring of $G$ as long as $q\geq \Delta+\Omega(1).$

Let $X$ be the $d$-partite complex defined as follows: For every vertex $v$, $S(v)$ is the list of colors available to $v$, facets of $X$ correspond to all proper colorings of $G$, and $\mu$ is the uniform distribution over all facets. Let $A$ be the adjacency matrix of $G$; attach a boundary vertex $b_v$ to every $v$ and let $P$ be the simple random walk matrix where $W(u,v)=\frac{1-\eps}{\text{deg}(u)}$ for any $u\sim v$, and $W(u,b_u)=\eps$ for all $u$.
Having the random walk matrix $W$, we can recursively define $p_\sigma$ for any $\sigma\in X$ using 
 \eqref{eq:linearpcase}. 
\begin{conjecture}
There is a constant $C>0$ such that if $|S(v)| \geq \text{deg}(v)+C$ for all $v\in V$, then $p_\varnothing$ is $\cC_\varnothing$-Lorentzian.
\end{conjecture}

For example, if $G$ is just one edge then the above conjecture holds for $C\geq 3$.
A few remarks are in order:
\begin{itemize}
    \item If the above conjecture holds, then by the argument in \cref{sec:dir-derivs} and proof of \cref{thm:maintechnical} we obtain that $\lambda_2(P_\varnothing)\leq O(\frac{1}{\eps(d-1)}).$
    \item We don't expect a trickledown argument that only uses second eigenvalues of links of codimension 2 lead to a proof of the above conjecture. This is because in the case of a complete graph on $d$ vertices (and same set of colors) all links of codimension 2 have second eigenvalue $\frac{1}{q-\Delta}$. In general the trickle-down \cref{thm:oppenheim} is tight when all links of codimension 2 have the same eigenvalue (in the worst case).
    \item \cref{thm:conn+quad=Lor} and \cref{thm:maintechnical}  only give sufficient conditions for $p_\sigma$ to be $\cC_\sigma$-Lorentzian. In general, it could be that $p_\sigma$ is $\cC_\sigma$-Lorentzian for significantly smaller values of $q\ll 2\Delta$. In such a case we need to find new techniques to verify whether a given polynomial is $\cC$-Lorentzian.
\end{itemize}

\printbibliography
\appendix
\section{Pairwise Influence vs the Spectral Pairwise 
Influence Matrices}
\label{app:IcI}
In this section we prove \cref{lem:IvscI}. We require the following lemma, which we will prove after proving \cref{lem:IvscI}.
\begin{lemma}
\label{lem:abgeometricmean}
    Let $P = \begin{bmatrix}
        0 & A \\ B & 0
    \end{bmatrix}$ be a real matrix with constant row sums $a$ and suppose $\lambda$ is an eigenvalue of $P$ other than $\pm a$. If $A$ is $m \times n$ with entries $a_{ij}$ and $B$ is $n \times m$ with entries $b_{ij}$ then
    \[
        |\lambda| \leq \frac{1}{2} \sqrt{\left(\max_{1 \leq i,j \leq m} \sum_{s=1}^n |a_{is}-a_{js}|\right) \left(\max_{1 \leq i,j \leq n} \sum_{s=1}^m |b_{is}-b_{js}|\right)}.
    \]
\end{lemma}
\begin{proof}[Proof of \cref{lem:IvscI}]
Let $I$ be the influence matrix of $\mu$. %Then, since all entries of $I$ are non-negative, by the Perron-Frobenius theorem, $\rho(I)=\lambda_{\max}(I).$
By \cref{lem:rhoaverage}, 
$$ \lambda_{\max}(\bar{I})=\rho(\bar{I})\leq \rho(I).$$
Next, we show that for any $u,v$,
\begin{equation}\label{eq:cIvsbarI} \cI(u,v)\leq \bar{I}(u,v).
\end{equation}
Note that this finishes the proof as by \cref{lem:monotoneeigenvalues} it would imply that $\lambda_{\max}(\cI)\leq\lambda_{\max}(\bar{I}).$
For a pair of vertices $u,v$ consider the worst link $\tau$ of codimension 2 such that $u,v\notin V(\tau)$ in the sense that $\lambda_2(P_\tau)=\cI(u,v).$ Such a link exists by the definition of the spectral pairwise influence matrix $\cI$.

Now we can write $P_\tau=\begin{bmatrix}0&A\\B&0\end{bmatrix}$ 
as a matrix with non-negative entries, such that every row sum is 1. 
Let $a_{s,s'}$ (resp. $b_{s,s'}$) be entries of $A$ (resp. $B$). 
Suppose the rows of $A$ correspond to spins of $u$ and the rows of $B$ correspond to the spins of $v$.
It then follows by the definition of the influence matrix $I$ that 
$$ \frac12 \left(\max_{1 \leq i,j \leq m} \sum_{s=1}^n |a_{is}-a_{js}|\right)\leq I_{u\to v} \quad \quad \frac12 \left(\max_{1 \leq i,j \leq n} \sum_{s=1}^m |b_{is}-b_{js}|\right)\leq I_{v\to u}.$$
Since the skeleton $G_\tau$ of $X_\tau$ is connected, we have that $\lambda_2(P_\tau) \neq 1$. Thus either 
$$\lambda_2(P_\tau) < 0=\cI(u,v)\leq \bar{I}(u,v),$$  or else by \cref{lem:abgeometricmean} we have
$$ \cI(u,v)=\lambda_2(P_\tau) \leq  \frac{1}{2} \sqrt{\left(\max_{1 \leq i,j \leq m} \sum_{s=1}^n |a_{is}-a_{js}|\right) \left(\max_{1 \leq i,j \leq n} \sum_{s=1}^m |b_{is}-b_{js}|\right)} \leq \bar{I}(u,v).$$
This proves \eqref{eq:cIvsbarI} as desired.
\end{proof}

We now prove \cref{lem:abgeometricmean}, which is a bipartite version of a result found in \cite{seneta1981non}. That said, we will make use of the following intermediate lemma from \cite{seneta1981non} for the proof.

\begin{lemma}[Lemma 2.4 of \cite{seneta1981non}]
    Suppose $\bm\delta \in \R^n$, $n \geq 2$, $\bm\delta^\top \bm{1} = 0$, $\bm\delta \neq 0$. Then for a suitable set $S = S(\bm\delta)$ of ordered pairs of indices $(i,j)$ with $1 \leq i,j \leq n$, we have
    \[
        \bm\delta = \sum_{(i,j) \in S} \left(\frac{\eta_{ij}}{2}\right) \cdot \gamma(i,j),
    \]
    where $\eta_{ij} > 0$ and $\sum_{(i,j) \in S} \eta_{ij} = \|\bm\delta\|_1$, and $\gamma(i,j) = \bm{e}_i - \bm{e}_j$ where $\bm{e}_i$ is the $i$th standard basis vector.
\end{lemma}

\begin{proof}[Proof of \cref{lem:abgeometricmean}]
    Let $\bm{z}^\top = (z_1,\ldots,z_n)$ be an arbitrary row vector of complex numbers. For any real $\bm\delta \neq \bm{0}$ with $\bm\delta^\top \bm{1} = 0$, the previous lemma implies
    \begin{equation}\label{eq:zdelta}
        |\bm{z}^\top \bm\delta| \leq \sum_{(i,j) \in S} \left(\frac{\eta_{ij}}{2}\right) |z_i-z_j| \leq \frac{1}{2} \max_{i,j} |z_i-z_j| \cdot \|\bm\delta\|_1.
    \end{equation}
    Note that the first term being at most the last term also trivially holds for $\bm\delta = 0$.
    Letting $f(\bm{z}) = \max_{i,j} |z_i-z_j|$, 
    \[
        f(A\bm{z}) = \max_{i,j} \left|\sum_{s=1}^n (a_{is}-a_{js})z_s\right| \leq \frac{1}{2} f(\bm{z}) \max_{i,j} \sum_{s=1}^n |a_{is}-a_{js}|.
    \]
    To see the inequality
    let $\bm\delta$ be the vector in $\R^n$ given by the difference of the $i$th and $j$th rows of $A$. Then $\bm\delta^\top \bone = 0$, and so the inequality follows by \eqref{eq:zdelta}.
    %let $\bm\delta$ be the vector in $\R^{2n}$ where the first $n$ coordinates are $a_{is}$'s and the second $n$ coordinates are $a_{js}$'s. Then, $\bm\delta^top \bm{1}=0$ and so the inequality follows by \eqref{eq:zdelta}.
    
    Since we have a similar inequality for $f(B\bm{y})$ for any complex row vector $\bm{y}^\top = (y_1,\ldots,y_m)$, we obtain
    \[
        f(AB\bm{y}) \leq \frac{1}{2} f(B\bm{y}) \max_{i,j} \sum_{s=1}^n |a_{is}-a_{js}| \leq \frac{1}{4} f(\bm{y}) \left(\max_{i,j} \sum_{s=1}^n |a_{is}-a_{js}|\right) \left(\max_{i,j} \sum_{t=1}^m |b_{it}-b_{jt}|\right).
    \]
    Now let $\bm{x}^\top = (\bm{y}^\top, \bm{z}^\top) = (y_1,\ldots,y_m,z_1,\ldots,z_n)$ be a right eigenvector of $P$ corresponding to an eigenvalue $\lambda$ of $P$, so that $AB\bm{y} = \lambda A\bm{z} = \lambda^2 \bm{y}$. We thus have
    \[
        |\lambda|^2 f(\bm{y}) = f(AB\bm{y}) \leq \frac{1}{4} f(\bm{y}) \left(\max_{i,j} \sum_{s=1}^n |a_{is}-a_{js}|\right) \left(\max_{i,j} \sum_{t=1}^m |b_{it}-b_{jt}|\right).
    \]
    Now $\lambda \neq \pm a$ implies both $\bm{y}$ and $\bm{z}$ are not multiples of the all-ones vector. Further, at least one of $\bm{y}$ and $\bm{z}$ is non-zero, and we can assume it is $\bm{y}$ without loss of generality by possibly swapping $A$ and $B$. Thus $f(\bm{y}) \neq 0$, and therefore
    $$|\lambda|^2 \leq \frac14 \left(\max_{i,j} \sum_{s=1}^n |a_{is}-a_{js}|\right) \left(\max_{i,j} \sum_{t=1}^m |b_{it}-b_{jt}|\right). $$
    The statement follows by taking the square root of both sides. 
\end{proof}

\section{Sum Graphs}
Let $G=(L,R,E)$ be a bipartite graph with $L=\{0,1,\dots,n_L-1\}, R=\{n_L,n_L+1,\dots,n_L+n_R-1\}$ for $n_L,n_R\leq C+1$. Let $w:E\to\R$ defined as follows: For $i\in L, j\in R$ with $i+j-n_L\leq C$ there is an edge with weight $w(i,j)=\lambda^{i+j-n_L}.$

\begin{lemma} \label{lem:multi-state-codim2}
    If $\lambda<1$, then  $\lambda_2(P_G)\leq \frac{\lambda^{(C+1)/2}}{1-\sqrt{\lambda}}$, where $P_G$ is the transition probability matrix of the simple random walk on $G$
\end{lemma}
\begin{proof}
    Let $A$ be the adjacency matrix of $G$ and $D$ be the degree matrix. Note that,
    \begin{equation}\label{eq:deglowerbound} d_w(i) \geq \lambda^i, ~~ \forall i\in L, \quad\quad d_w(j) \geq \lambda^{j-n_L}, ~~ \forall j\in R.\end{equation}
    We need to show that
    \begin{equation} \label{eq:goalADlambda}A\preceq vv^T + \frac{\lambda^{(C+1)/2}}{1-\sqrt{\lambda}}D\end{equation}
    for some vector $v\in \R^{L+R}$.
    
    Let $v_L\in \R^{L}, v_R\in \R^R$ defined as follows: For any $i\in L$, $v_L(i)=\lambda^i$ and for any $j\in R$, $v_R(j)=\lambda^{j-n_L}$. Let
    $$ B:=A-(v_L+v_R)(v_L+v_R)^T +v_Lv_L^T + v_Rv_R^T$$

    To prove \eqref{eq:goalADlambda}, it is enough to show that 
\begin{equation}\label{eq:goalADlambda2}B\preceq \frac{\lambda^{(C+1)/2}}{1-\sqrt{\lambda}}D.\end{equation}
    First, we observe that
$$ B=\sum_{i\in L,j\in R: i+j-n_L>C} B_{i,j},$$
where $B_{i,j}$ is only supported on the $(i,j)$ and $(j,i)$ entries and it is the same as $B$ on those entries. %Further, notice $B=\sum_{i,j:i+j-n_L>C} -B_{i,j}$.
 
%Define a function $f_L:L\times R$ as follows:
%\begin{align*} f_L(i,j)&=\lambda^{1.5(i+j-n_L)}\I[i>C/2] + \lambda^{0.5(i+j-n_L)}\I[i\leq C/2]\\
%\end{align*}
%We define $f_R(i,j)$ similarly by looking at whether $j-n_L>C/2$.
By \cref{lem:2by2-ineq} below we can write
$$ B_{i,j}\preceq \lambda^{i+\frac{i+j-n_L}{2}}I_{i} + \lambda^{j-n_L + \frac{i+j-n_L}{2}}I_{j},$$
where $I_{i}\in \R^{(L+R)\times (L+R)}$ is a diagonal matrix with $\bone_i$ on the diagonal. 
Having this, we write,
$$ B\preceq \sum_{i,j: i+j-n_L>C}\left(\lambda^{i+\frac{i+j-n_L}{2}}I_{i} + \lambda^{j-n_L + \frac{i+j-n_L}{2}}I_{j}\right) \preceq \frac{\lambda^{(C+1)/2}}{1-\sqrt{\lambda}}D$$
To see the last inequality, notice that for any $i\in L$ we have
%\textcolor{red}{We made a mistake below we have $(i-(j-n_L))/C$ in the exponent not $(i+j-n_L)/C$}
\[
    \sum_{j\in R:i+j-n_L>C} \lambda^{i+\frac{i+j-n_L}{2}} \leq \lambda^i\cdot \left(\lambda^{\frac{C+1}{2}}+\lambda^{\frac{C+2}{2}}+\dots \right) = \lambda^i \cdot \frac{\lambda^{(C+1)/2}}{1-\sqrt{\lambda}}
\]
where we used \eqref{eq:deglowerbound} in the last inequality. A similar inequality holds for any $j\in R$. This proves \eqref{eq:goalADlambda2} as desired. 
\end{proof}

% \begin{align*}
%     \sum_{j:i+j-n_L > C} \lambda^{i+\frac{i+j-n_L}{2}} &\leq \lambda^i\cdot \left(\lambda^{\frac{C+1}{2}}+\lambda^{\frac{C+2}{2}}+\dots + \lambda^{\frac{i+n_R-1}{2}} \right) \\
%         &= \frac{\lambda^{i+(C+1)/2}(1-\lambda^{\frac{i+n_R-C-1}{2}})}{1-\sqrt{\lambda}} \\
%         &\leq \frac{\lambda^{(C+1)/2}(1+\sqrt{\lambda})(1-\lambda^{\frac{i}{2}})}{1-\lambda^{\frac{C-i+1}{2}}} \cdot \frac{\lambda^i(1-\lambda^{\frac{C-i+1}{2}})}{1-\lambda}
% \end{align*}

% Note the following for the above computations.
% We want
% \[
%     \frac{i+n_R-C-1}{2} \leq \min\{n_R,C-i+1\},
% \]
% equivalent to
% \[
%     n_R - (C-i+1) \leq 2\min\{n_R,C-i+1\}.
% \]
% If $n_R < C-i+1$ then we are done. Otherwise we want
% \[
%     n_R - (C-i+1) \leq 2(C-i+1)
% \]
% Since $n_R \leq C+1$ we have
% \[
%     n_R - (C-i+1) \leq C+1 - (C-i+1) = i
% \]

\begin{lemma} \label{lem:2by2-ineq}
    For any $\alpha \in \R$ and for any $a,b \geq 0$ such that $a\cdot b=\alpha^2$ we have
    $$ \begin{bmatrix} 0 & \alpha \\ \alpha & 0\end{bmatrix} \preceq \begin{bmatrix} a & 0 \\ 0 & b\end{bmatrix}$$
\end{lemma}
\begin{proof}
    The statement simply follows from when we subtract the left from the right  we get a  non-negative diagonal with a 0 determinant matrix rank 2 matrix.
\end{proof}

\end{document}